\documentclass[
reqno,
11pt,
oneside
]{article}

\usepackage[ngerman, english]{babel}
\usepackage[utf8]{inputenc}
\usepackage[normalem]{ulem}
\usepackage[babel,french=guillemets,german=swiss]{csquotes}
\usepackage[center,font=small]{caption}
\usepackage{cite}
\usepackage{bbm}
\usepackage[raggedright]{titlesec}
\usepackage{xcolor}
\usepackage{enumerate}
\usepackage{setspace}

\titleformat{\section}{\normalfont\large\bfseries}{\thesection}{1em}{}
\titleformat{\subsection}{\normalfont\bfseries}{\thesubsection}{1em}{}

\usepackage{tocbasic}
\DeclareTOCStyleEntry[
beforeskip=.2em plus 1pt,
]{tocline}{section}

\usepackage[%
left=1.5in,
right=1.5in,
bottom=1in,
top=0.8in,
footskip=0.4in,
]{geometry}

\usepackage{color}
\definecolor{LinkColor}{rgb}{0,0,1}
\definecolor{LinkColor2}{rgb}{0,0.5,0}
\definecolor{lbcolor}{rgb}{0.85,0.85,0.85}
\definecolor{FrameColor}{rgb}{0.85,0.85,0.85}
\definecolor{rosso}{rgb}{0.8,0,0}
\definecolor{lightgray}{rgb}{0.5,0.5,0.5}
\definecolor{violet}{rgb}{0.65,0,0.65}
\definecolor{darkgreen}{rgb}{0,0.5,0}

\usepackage{enumitem}

\usepackage{graphicx}
\usepackage{placeins}
\usepackage{overpic}

\usepackage{amsmath}
\usepackage{amssymb}
\usepackage{dsfont}
\usepackage{empheq}
\usepackage{amsthm}

\usepackage[%
pdftitle={Titel},%
pdfauthor={Autor},%
pdfcreator={LaTeX, LaTeX with hyperref and KOMA-Script},
pdfsubject={Betreff}, 
pdfkeywords={Keywords}
]{hyperref} 

\hypersetup{%
	colorlinks	=true,
	linkcolor	=LinkColor,%
	anchorcolor	=LinkColor,%
	citecolor	=LinkColor2,%
	filecolor	=LinkColor,%
	menucolor	=LinkColor,%
	urlcolor	=LinkColor,%
}
\usepackage{marginnote}

\usepackage{verbatim}

\newtheorem{theorem}{Theorem}[section]
\newtheorem{lemma}[theorem]{Lemma}
\newtheorem{proposition}[theorem]{Proposition}
\newtheorem{corollary}[theorem]{Corollary}

\theoremstyle{definition}
\newtheorem{remark}[theorem]{Remark}

\makeatletter
\renewenvironment{proof}[1][\proofname]{%
	\par\pushQED{\qed}\normalfont%
	\topsep6\p@\@plus6\p@\relax
	\trivlist\item[\hskip\labelsep\bfseries#1\@addpunct{.}]%
	\ignorespaces
}{%
	\popQED\endtrivlist\@endpefalse
}
\makeatother

\makeatletter
\renewcommand\paragraph{\@startsection{paragraph}{4}{\z@}%
	{1ex \@plus1ex \@minus.2ex}%
	{-1em}%
	{\normalfont\normalsize\bfseries}}
\renewcommand\subparagraph{\@startsection{paragraph}{4}{\z@}%
	{1ex \@plus1ex \@minus.2ex}%
	{-1em}%
	{\normalfont\normalsize\itshape}}
\makeatother

\newcommand{\abs}[1]{\left| #1 \right|}
\newcommand{\norm}[1]{\| #1 \|}

\newcommand{\R}{\mathbb R}
\newcommand{\N}{\mathbb N}
\newcommand{\n}{\mathbf{n}}

\newcommand{\intO}{\int_\Omega}
\newcommand{\supp}{\textnormal{supp\,}}
\newcommand{\dist}{\textnormal{dist\,}}

\newcommand{\dx}{\;\mathrm{d}x}
\newcommand{\dy}{\;\mathrm{d}y}
\newcommand{\dhy}{\;\mathrm{d}\hat{y}}
\newcommand{\dhx}{\;\mathrm{d}\hat{x}}
\newcommand{\dz}{\;\mathrm{d}z}
\newcommand{\dt}{\;\mathrm dt}
\newcommand{\dw}{\;\mathrm dw}
\newcommand{\ds}{\;\mathrm ds}

\newcommand{\Div}{\operatorname{div}}

\newcommand{\wto}{\rightharpoonup}

\newcommand{\LL}{\mathcal{L}}

\newcommand{\eps}{\varepsilon}

\begin{document}
	



	
\title{\bfseries 
    Nonlocal-to-local $L^p$-convergence of \\ convolution operators with singular, anisotropic kernels 
}

\author{
    Helmut Abels \footnotemark[1]
     \and Christoph Hurm \footnotemark[1]
     \and Patrik Knopf \footnotemark[1] \footnotemark[2]
    }

\date{ }

\maketitle

\renewcommand{\thefootnote}{\fnsymbol{footnote}}

\footnotetext[1]{
    Faculty of Mathematics, 
    University of Regensburg, 
    D-93053 Regensburg, 
    Germany \newline
	\tt(%
        \href{mailto:helmut.abels@ur.de}{helmut.abels@ur.de},
        \href{mailto:christoph.hurm@ur.de}{christoph.hurm@ur.de},
        \href{mailto:patrik.knopf@ur.de}{patrik.knopf@ur.de}%
        ).
}

\footnotetext[2]{
    Institute for Analysis,
    Department of Mathematics,
    Karlsruhe Institute of Technology (KIT), 
    D-76128 Karlsruhe, 
    Germany \newline
	\tt(%
        \href{mailto:patrik.knopf@kit.edu}{patrik.knopf@kit.edu}%
        ).
}

\vspace{-2ex}
\begin{center}
	\scriptsize
	\color{white}
	{
		\textit{This is a preprint version of the paper. Please cite as:} \\  
		C.~Hurm
        \textit{[Journal]} \textbf{xx}:xx 000-000 (2026), \\
		\texttt{https://doi.org/\ldots}
	}
\end{center}
\vspace{-2ex}

\medskip


\begin{small}
\begin{center}
\begin{minipage}[h]{0.9\textwidth}
    \textbf{Abstract.}
    We study nonlocal convolution-type operators with singular, possibly anisotropic kernels. Our main objective is to establish and quantify their nonlocal-to-local convergence to a local differential operator with natural boundary conditions, as the kernels concentrate at the origin in a suitable way. Such convergence results provide a useful tool for the physical justification of mathematical models, particularly in situations where the desired local differential operator cannot be directly derived from microscopic laws. The present work substantially extends previous results by allowing kernels with stronger singularities (comparable to those of fractional Laplacians), anisotropic and non-localized kernels, and by proving strong convergence in general $L^p$ spaces together with explicit convergence rates.
\\[2ex]
\\[2ex]	

\end{minipage}
\end{center}
\end{small}

\bigskip\bigskip


\begin{small}
\setcounter{tocdepth}{2}
\hypersetup{linkcolor=black}
\tableofcontents
\end{small}

\setlength\parskip{1ex}
\allowdisplaybreaks[0]
\numberwithin{equation}{section}
\renewcommand{\thefootnote}{\arabic{footnote}}

\newpage




\section{Introduction} 
\label{SECT:INTRO}

We consider nonlocal convolution type operators of the form
\begin{align}
    \label{DEF:L:BASIC}
    \begin{split}
    \LL^\Omega_\ast u(x) 
    &= \mathrm{P.V.} \intO J(x-y) \big( u(x) - u(y) \big) \dy 
    \\
    &= \underset{r\searrow 0}{\lim}\,\int_{\Omega\cap\{\abs{x-y}\ge r\}} J(x-y) \big( u(x) - u(y) \big) \dy
    \end{split}
\end{align}
for all $x\in\Omega$ and any suitable function $u:\Omega\rightarrow\R$,
where $\Omega\subseteq\R^n$ is a sufficiently smooth domain and 
\begin{equation}
    \label{DEF:J:INTRO}
    J:\R^n\setminus\{0\} \to \R, \quad J(x) = \frac{\rho(x)}{\abs{x}^2},
\end{equation}
is a prescribed interaction kernel.
Throughout this paper, the function $\rho:\R^n\to\R$ is supposed to be even and has to fulfill certain integrability conditions that will be specified in Section~\ref{SECT:PRELIM}. In particular, the convolution kernel $J$ is also even. Because of this symmetry, the nonlocal energy associated with $\LL^\Omega_\ast$ is given by
\begin{align*}
    \mathcal{E}_\ast^\Omega(u)
    &= \intO \LL^\Omega_\ast u(x) \; u(x) \dx
    = \frac 12 \intO\intO J(x-y) \big| u(x) - u(y) \big|^2 \dx
\end{align*}
for any suitable function $u:\Omega\rightarrow\R$.

\paragraph{Main objective of the paper.} 
The main goal of this paper is to study the \textit{nonlocal-to-local convergence} of nonlocal operators as introduced in \eqref{DEF:L:BASIC}. 
This means that we intend to show that such a nonlocal operator converges to a local differential operator in some suitable sense as the function $\rho$ approaches a multiple of the $\delta$-distribution. To formulate this in a mathematically precise way, we define the functions
\begin{alignat}{2}
    \label{DEF:RHOE:INTRO}
    &\rho_\eps:\R^n \to \R, 
    &&\quad \rho_\eps(x) \coloneqq \eps^{-n}\, \rho\big( \tfrac{x}{\eps} \big),
    \\
    \label{DEF:JE:INTRO}
    &J_\eps:\R^n\setminus\{0\} \to \R, 
    &&\quad J_\eps(x) = \frac{\rho(x)}{\abs{x}^2}
\end{alignat}  
for any $\eps> 0$. This means that $(\rho_\eps)_{\eps>0}$ is a \textit{Dirac sequence}, i.e., it approaches a multiple of the $\delta$-distribution in the limit $\eps\to 0$ (cf.~Lemma~\ref{LEM:DIRAC}). For any $\eps>0$, we write
\begin{align}
    \label{DEF:LE}
    \begin{split}
    \LL_\eps^\Omega u(x) \coloneqq \mathrm{P.V.} \intO J_\eps(x-y) \big( u(x) - u(y) \big) \dy 
    \end{split}
\end{align}
for all $x\in\Omega$, to denote the nonlocal operator corresponding to $\rho_\eps$.
We want to show that in the limit $\eps\to 0$, the nonlocal operator $\LL_\eps$ converges to a local differential operator $\LL$ of the type
\begin{align}
    \label{INTRO:DEF:LL}
    \LL^\Omega u(x) \coloneqq - \Div(M\nabla u) = - \sum_{i,j=1}^n M_{ij}\, \partial_{x_j}\partial_{x_k}u(x)
    \quad
    \text{for all $x\in \Omega$},
\end{align}
where $M\in\R^{n\times n}$ is a symmetric matrix with entries
\begin{align}\label{INTRO:DEF:M}
    M_{ij} \coloneqq \frac{1}{2} \int_{\R^n}J(z)\, z_i\, z_j\dz
    \quad\text{for all $i,j\in\{1,\ldots,n\}$}. 
\end{align}
$M$ is usually referred to as the \textit{momentum matrix}. In case $\Omega \subsetneq\R^n$, which means that $\partial\Omega\neq \emptyset$, we interpret the second-order differential operator $\LL^\Omega$ to be equipped with the natural boundary condition
\begin{equation}
    \label{INTRO:BC}
	M \nabla u \cdot \n_{\partial\Omega} = 0 \quad\text{on $\partial\Omega$}.
\end{equation}
Here, $\n_{\partial\Omega}$ denotes a unit normal vector field on $\partial\Omega$.
Under suitable assumptions on the domain $\Omega$ and on $\rho$, $J$ and $p$, we intend to show that
\begin{equation}
    \label{INTRO:CONV}
    \LL_\eps^{\Omega} u \to \LL^{\Omega} u 
    \quad\text{in $L^p(\Omega)$ as $\eps\to 0$}
\end{equation}
for all functions $u\in W^{2,p}(\Omega)$ with \eqref{INTRO:BC} if $\Omega \subsetneq\R^n$.
In addition, we are planning to establish the convergence rate
\begin{equation}
    \label{INTRO:CONV:RATE}
    \big\|\LL_\eps^{\Omega} u - \LL^{\Omega} u \big\|_{L^p(\Omega)} 
    \leq C_\Omega\, \sqrt[p]{\eps}\, \|u\|_{W^{3,p}(\Omega)}.
\end{equation}
for all functions $u\in W^{3,p}(\Omega)$ with \eqref{INTRO:BC} if $\Omega \subsetneq\R^n$.
Here, $C$ is a positive constant that depends only on $\Omega$, $\rho$ and $p$.

\paragraph{State of the art and related results.}
The nonlocal-to-local convergence of $\LL_\eps^{\Omega}$ to $\LL^{\Omega}$ was first established in the weak sense. In this case, it can be derived from the convergence of the corresponding energies. The energy functional associated with the nonlocal operator $\LL^\Omega_\eps$ is given by
\begin{align*}
    \mathcal{E}_\eps^\Omega(u)
    &= \intO \LL^\Omega_\eps u(x) \; u(x) \dx
    = \frac 12 \intO\intO J_\eps(x-y) \big| u(x) - u(y) \big|^2 \dx
    \\
    &= \frac 12 \intO\intO \rho_\eps(x-y) \left(\frac{| u(x) - u(y) |}{|x-y|}\right)^2 \dx.
\end{align*}
and the energy functional associated with the local differential operator $\LL^\Omega$ reads as
\begin{align*}
    \mathcal{E}^\Omega(u)
    = \frac 12 \intO M \nabla u(x) \cdot \nabla u(x) \dx.
\end{align*}
Given these energies, one may intuitively expect that the weighted squared difference quotient $\rho_\eps(x-y) \big(\tfrac{|u(x)-u(y)|}{|x-y|}\big)^2$ approaches an expression of the form $M \nabla u(x) \cdot \nabla u(x)$ as the kernel $\rho_\eps$ concentrates at the origin, since in this regime only points $y$ close to $x$ contribute significantly to the integral. Here, the matrix $M$ accounts for anisotropic effects when $\rho$ is even but not radially symmetric. In case $\rho$ is radially symmetric, we expect $M$ to be a scalar multiple of the identity matrix and thus, $\LL^\Omega$ is a scalar multiple of the negative Laplacian.

A rigorous proof of the nonlocal-to-local convergence of such energies in a more general setting was provided by \textsc{Ponce} \cite{Ponce2004Gamma} (see also \cite{Ponce2004Poincare,Ponce2003Poincare}). 
These investigations of nonlocal energies provide an extension of earlier works by \textsc{Bourgain}, \textsc{Brezis} and \textsc{Mironescu} \cite{Bourgain2001}. Some improvements were obtained later, for example, by \textsc{Brezis} and \textsc{Nguyen} \cite{Brezis2016}. 
In particular, it was shown in \cite{Ponce2004Gamma} that
\begin{align}
    \label{CONV:EN:PTW}
    \mathcal{E}_\eps^\Omega(u) \to \mathcal{E}^\Omega(u)
    \quad\text{as $\eps\to 0$ for all $u\in H^1(\Omega)$.}
\end{align}
Moreover, also the $\Gamma$ convergence of such energies was discussed in \cite{Ponce2004Gamma}. Further results on nonlocal-to-local $\Gamma$-convergence were found by \textsc{Brezis} and \textsc{Nguyen} \cite{Brezis2018,Brezis2020}.

\pagebreak[3]

From the pointwise convergence \eqref{CONV:EN:PTW}, the weak convergence
\begin{equation}
    \label{CONV:WEAK:L2}
    \LL_\eps^{\Omega} u \wto \LL^{\Omega} u 
    \quad\text{weakly in $H^1(\Omega)'$ as $\eps\to 0$ for all $u\in H^1(\Omega)$}
\end{equation}
was derived by \textsc{Melchoinna}, \textsc{Ranetbauer}, \textsc{Scarpa} and \textsc{Trussardi} \cite{Melchionna2019} for radially symmetric, sufficiently regular kernels and $\Omega$ being a flat torus. An analogous result for bounded domains $\Omega\subset\R^n$ and radially symmetric kernels $J_\eps\in W^{1,1}_\mathrm{loc}(\R^n)$ was established by \textsc{Davoli}, \textsc{Scarpa} and \textsc{Trussardi} \cite{Davoli2021}. The main goal of these works was to relate the \textit{nonlocal Cahn--Hilliard equation} to the classical \textit{(local) Cahn--Hilliard equation} by a nonlocal-to-local limit. 
This is an important advance because, as shown by \textsc{Giacomin} and \textsc{Lebowitz} \cite{Giacomin1996} (see also \cite{Giacomin1997,Giacomin1998}), the local Cahn--Hilliard equation, in contrast to its nonlocal counterpart, cannot be derived directly from microscopic physical principles. 
In this sense, the nonlocal-to-local convergence offers a physical justification for the local Cahn--Hilliard equation. 
Therefore, besides \cite{Melchionna2019,Davoli2021}, the weak convergence \eqref{CONV:WEAK:L2} for radially symmetric $W^{1,1}_\mathrm{loc}$-kernels has been applied to variants of the Cahn--Hilliard equation and related models in various settings. We refer, for example, to \cite{AbelsTerasawa2022,Carrillo2025,Colli2024,Davoli2021convective,Elbar2023,Elbar2025,Kurima}. 

However, the aforementioned results are merely qualitative and do not yield any quantitative information on the nonlocal-to-local convergence, such as convergence rates. Also, at least for more regular functions $u$, it may be expected that the convergence $\LL_\eps^{\Omega} u \to \LL^{\Omega} u$ actually holds in a strong topology. A first step in this direction was achieved in \cite{AbelsHurm2024} by the first two authors of the present paper. For radially symmetric, localized (i.e., compactly supported) $W^{1,1}$-kernels and a sufficiently regular bounded domain $\Omega$, it was shown in \cite{AbelsHurm2024} that
\begin{align}
    \label{CONV:STRONG:L2}
    \LL_\eps^{\Omega} u \to \LL^{\Omega} u 
    \quad\text{strongly in $L^2(\Omega)$ as $\eps\to 0$}
\end{align}
 for all $u\in H^2(\Omega)$ with $\nabla u\cdot \n_{\partial\Omega} = 0$ on $\partial\Omega$, and
\begin{align}
    \label{CONV:RATE:L2}
    \norm{\LL_\eps^{\Omega} u - \LL^{\Omega} u }_{L^2(\Omega)} \le C \eps^\gamma \norm{u}_{H^{3}(\Omega)}
\end{align}
for all $u\in H^3(\Omega)$ with $\nabla u\cdot \n_{\partial\Omega} = 0$ on $\partial\Omega$.
Here, $C$ denotes a positive constant that may depend on $\Omega$ but is independent of $\eps$ and $u$. 
The exponent $\gamma$ is given by $\gamma=1$ if $\Omega=\R^n$ or $\Omega=\mathbb T^n$ (a flat torus), and $\gamma=\tfrac 12$ if $\Omega$ is a sufficiently regular bounded domain.
In \cite{AbelsHurm2024}, these convergence results were also applied to prove nonlocal-to-local convergence of the Cahn--Hilliard equation and the Allen--Cahn equation. In subsequent works, this technique was applied in \cite{AbelsHurmMoser2025,Hurm2024,HurmMoser2025,DiPrimioHurm} to analyze further models related to the Cahn–Hilliard equation. 

Furthermore, we remark that the kernels in the nonlocal-to-local convergence results mentioned so far were always assumed to be radially symmetric. However, in concrete applications, this might not always be the case. For example, models related to crystal growth usually involve anisotropic interactions. We point out that in the contributions \cite{Giacomin1996,Giacomin1997,Giacomin1998} by \textsc{Giacomin} and \textsc{Lebowitz}, the interaction kernel was merely assumed to be even, but not radially symmetric. Also in the work \cite{Ponce2004Gamma} by \textsc{Ponce} on the convergence of nonlocal energies, radial symmetry of the kernel is not required. This leads to the conjecture that it should be possible to generalize the aforementioned nonlocal-to-local convergence results to anisotropic kernels.
A first step in this direction was achieved very recently by the first author of the present paper and \textsc{Terasawa} \cite{AbelsTerasawa2025} by establishing the weak nonlocal-to-local convergence for an anisotropic variant of the Cahn--Hilliard equation. It is also worth mentioning that the strong $W^{1,1}_\mathrm{loc}$-regularity assumption on the convolution kernel was also relieved there. Instead, weaker assumptions similar to those by \textsc{Gounoue}, \textsc{Kassmann} and \textsc{Voigt} \cite{Foghem2020} were imposed.

In the context of our work, we also want to mention the recent contribution \cite{BUNGERT2026114173} by \textsc{Bungert} and \textsc{del Teso}, where nonlocal-to-local convergence rates in the $W^{2,s}(\R^n)$-norm were established for the the Dirichlet problem to the fractional Laplacian $(-\Delta)^s$ as $s\rightarrow 1$. For a discussion of nonlocal Neumann problems in a broader sense with general measurable, nonnegative convolution kernels, we refer to the work \cite{Frerick2025} by \textsc{Frerick}, \textsc{Vollmann} and \textsc{Vu}.

\paragraph{Novelties and main results of our paper.}

In the present paper, we intend to provide a generalization of the aforementioned results on nonlocal-to-local convergence.
Under suitable assumptions, we will show the convergences
\begin{align}
    \label{CONV:STRONG:LP}
    \,\LL_\eps^{\Omega} u \to \LL^{\Omega} u 
    \quad\text{strongly in $L^p(\Omega)$ as $\eps\to 0$}
\end{align}
for all $u\in W^{2,p}(\Omega)$ with $M\nabla u\cdot \n_{\partial\Omega} = 0$ on $\partial\Omega$, and
\begin{align}
    \label{CONV:RATE:LP}
    \norm{\LL_\eps^{\Omega} u - \LL^{\Omega} u }_{L^p(\Omega)} \le C \eps^{\gamma(p)} \norm{u}_{W^{3,p}(\Omega)}
\end{align}
for all $u\in W^{3,p}(\Omega)$ with $M\nabla u\cdot \n_{\partial\Omega} = 0$ on $\partial\Omega$.
Here, $p \in [1,\infty)$ and
\begin{align*}
    \gamma(p) =
    \begin{cases}
        1 &\text{if $\Omega = \R^n$,}
        \\
        \frac{1}{p} &\text{if $\Omega\subsetneq \R^n$ is a sufficiently regular domain.}
    \end{cases}
\end{align*}

Compared with the previous results in \cite{AbelsHurm2024}, our new results provide the following improvements:
\begin{enumerate}[label=\textbf{(\roman*)}, topsep=0.5ex, itemsep=0.25ex, leftmargin=*]
    \item \label{IMP:1} The domain $\Omega$ does not need to be bounded, but we merely require $\Omega$ to have a sufficiently regular, compact boundary. This means that exterior domains (e.g., $\R^n\setminus B_1(0)$) can also be considered.

    \item \label{IMP:2} The strong nonlocal-to-local convergence is established in a general $L^p$ framework. This enhances the applicability of the results, as in concrete applications one may need convergence in an $L^p$ space other than $L^2$, or the relevant function $u$ may lie only in $W^{k,p}$ instead of $H^k$ for $k\in\{2,3\}$.
    Note that the rate obtained in \eqref{CONV:RATE:LP} is consistent with the one from \cite{AbelsHurm2024} (cf.~\eqref{CONV:RATE:L2}) for $p=2$.

    \item \label{IMP:3} We demand significantly weaker assumptions on $\rho$ and $J$, see \ref{ASS:RHO} and \ref{ASS:J}. 
    In particular, $W^{1,1}_\mathrm{loc}$-regularity of $J_\eps$ is not required, 
    and we allow for a much stronger singularity at the origin, comparable to that of a fractional Laplacian $(-\Delta)^s$ of order $s\in (0,1)$ (cf.~Remark~\ref{REM:FRACLAP}).

    \item \label{IMP:4} As in \cite{AbelsTerasawa2025}, we also do not require $\rho$ to be radially symmetric. This allows for anisotropic effects, since different directions may be weighted differently, resulting in a non-diagonal momentum matrix $M$. 
    Compared to \cite{AbelsTerasawa2025}, our advance is that we establish the nonlocal-to-local convergence in the strong sense as well as associated convergence rates.
\end{enumerate}
Especially \ref{IMP:3} and \ref{IMP:4} give rise to considerable technical difficulties in our proofs.

\paragraph{Structure of the paper.}
In Section~\ref{SECT:PRELIM}, we first introduce our fundamental assumptions on the convolution kernel. We also give a concrete example of class of kernels, which satisfy these assumptions. Furthermore, we explain in Remark~\ref{REM:FRACLAP}, that our assumptions allow the kernels to exhibit singularities at the origin that are comparable to those of fractional Laplacians.

Our main results are successively stated and proved in Section~\ref{SECT:MAIN}. In Subsection~\ref{SUBSEC:WELLDEF}, we first show that the nonlocal operator \eqref{DEF:LE} is actually well defined under our assumptions on the convolution kernel. Then in Subsection~\ref{SUBSEC:FULLSPACE}, we establish our main result in the case $\Omega=\R^n$.
Afterwards, we prove the convergences \eqref{CONV:STRONG:LP} and \eqref{CONV:RATE:LP} in the special case that $\Omega$ is a sufficiently regular curved half-space. Finally, in Subsection~\ref{SUBSEC:DOMAIN}, this allows us to establish our main result in the general case that $\Omega\subset\R^n$ is a domain with sufficiently regular, compact boundary.




\section{Assumptions and preliminaries} \label{SECT:PRELIM}

In the remainder of this paper, let the dimension $n\in\N$ be arbitrary. We make the following general assumptions on the convolution kernel.

\begin{enumerate}[label=\textnormal{\bfseries (A\arabic*)}, topsep=0ex, leftmargin=*]
    \item\label{ASS:RHO} 
    We assume that $\rho:\R^n\to [0,\infty)$, $\rho\not\equiv 0$ is measurable and even (i.e., $\rho(x) = \rho(-x)$ for almost all $x\in\R^n$) with
    \begin{align}
        \label{COND:RHO:1}
        \int_{\R^n} \rho(x) (1+\abs{x}) \dx <\infty.
    \end{align}
    The associated kernel is given by
    \begin{equation}
        \label{DEF:J}
        J:\R^n\setminus\{0\} \to \R, \quad J(x) = \frac{\rho(x)}{\abs{x}^2}.
    \end{equation}
    For any $\eps>0$, we introduce the functions
    \begin{alignat}{2}
    	\label{DEF:RHOE}
    	&\rho_\eps:\R^n \to \R, 
    	&&\quad \rho_\eps(x) \coloneqq \eps^{-n}\, \rho\big( \tfrac{x}{\eps} \big),
    	\\
    	\label{DEF:JE}
    	&J_\eps:\R^n\setminus\{0\} \to \R, 
    	&&\quad J_\eps(x) = \frac{\rho_\eps(x)}{\abs{x}^2}.
    \end{alignat}  
    
    Since $\rho\geq 0$ and $\rho\not\equiv 0$, the momentum matrix 
    \begin{align}\label{DEF:M}
    	M \coloneqq
    	\frac{1}{2} \int_{\R^n} J(z)\, z\otimes z \dz
    \end{align}
    is positive definite.

    \item\label{ASS:J} In addition to Assumption~\ref{ASS:RHO}, we assume that $\rho\in C^1(\R^n\setminus\{0\})$ and that there exist $c_0,c_1>0$, $\alpha \in (0,2)$ with $N>3-\alpha$ and
    \begin{equation}
        \label{COND:RHO:2}
        \begin{split}
        |\rho(x)| &\le c_0  |x|^{2-\alpha-n}\, (1+|x|)^{-N}
        \\
        |\nabla\rho(x)| &\le c_1   |x|^{1-\alpha-n}\, (1+|x|)^{-N}
        \end{split}
    \end{equation}
    for all $x\in B_1(0)\setminus\{0\}$.
    Moreover, we assume that the associated kernel $J$ fulfills the conditions
    \begin{align}
    	\label{COND:J:2}
    	&\int_{\R^{n-1}} J\left(AQ
    	\begin{pmatrix}
    		z' \\ z_n
    	\end{pmatrix}\right) z' \,\mathrm{d}z' = 0
    	\quad\text{for all $z_n\in\R$ and all $Q\in \mathrm{SO}(n)$,}
    \end{align}
    where
    \begin{equation}
    	\label{DEF:A}
    	A \coloneqq \sqrt{M}.
    \end{equation}
\end{enumerate}

\bigskip

\begin{remark} \label{REM:ASS} \normalfont
    We make the following observations.
    \begin{enumerate}[label=\textnormal{(\alph*)}, topsep=0ex, leftmargin=*]
        \item \label{REM:ASS:1} If $\rho\in C^1(\R^n\setminus\{0\})$ with compact support in $\R^n$, it suffices to demand
        \begin{equation}
            \label{COND:RHO:2:C}
            \begin{split}
            |\rho(x)| &\le c_0  |x|^{2-\alpha-n}
            \\
            |\nabla\rho(x)| &\le c_1   |x|^{1-\alpha-n}
            \end{split}
        \end{equation}
        instead of \eqref{COND:RHO:2}. In this situation, \eqref{COND:RHO:2:C} implies that \eqref{COND:RHO:2} holds for every $N>0$, provided that $c_0$ and $c_1$ are chosen sufficiently large.
        
        \item \label{REM:ASS:2} If $\rho$ is radially symmetric, so is $J$. In this case, condition \eqref{COND:J:2} is automatically fulfilled. This is because $A = a\mathrm{I}$ for some $a>0$, where $\mathrm{I}$ denotes the identity matrix, which implies that
        \begin{align*}
            \int_{\R^{n-1}} J(AQz) z' \,\mathrm{d}z'
            = \int_{\R^{n-1}} J(az) z' \,\mathrm{d}z'
            = 0.
        \end{align*}
        Here, the second equality can be shown by applying the change of variables $z' \mapsto -z'$.
    \end{enumerate}
\end{remark}

\medskip

\begin{remark} \label{REM:FRACLAP} \normalfont
Recall that, in case $\Omega=\R^n$, the \textit{fractional Laplace operator} of order $s\in (0,1)$ can be represented as
\begin{align*}
    (-\Delta)_\Omega^{s} u(x) = \frac{4^s \Gamma\left(\tfrac n2 + s\right)}{\pi^{\frac n2} \abs{\Gamma(-s)} } \; \mathrm{P.V.}\!\int_{\Omega} \frac{\rho(x-y)}{|x-y|^2} \big(u(x) - u(y)) \dy,
\end{align*}
with
\begin{align}
    \label{DEF:RHO:FRAC}
    \rho(x) = |x|^{2-n-2s}
\end{align}
for all $x\in\R^n$, where $\Gamma$ denotes the Gamma function. If $\Omega \subsetneq \R^n$, the above operator is referred to as a \textit{regional fractional Laplace operator} of order $s$.
We notice that the singularity of $\rho$ at the origin is as demanded in \eqref{COND:RHO:2} with $\alpha = 2s$.
However, $\rho$ needs to be modified for large values of $|x|$ in a suitable way in order to satisfy \eqref{COND:RHO:1} and \eqref{COND:RHO:2}.
In case $\Omega$ is a bounded domain, such a modification for large $|x|$ does not constitute a true restriction. If $\Omega$ is bounded, we can find a radius $R>0$ such that $\Omega \subset B_R(0)$. This means that for any $x\in\Omega$, $(-\Delta)_\Omega^{s} u(x)$ does not depend on the values of $\rho$ on $\R^n\setminus B_{2R}(0)$. Therefore, we can simply replace $\rho$ by
\begin{align*}
    \tilde{\rho}(x) = \rho(x)\chi(x) = |x|^{2-n-2s} \chi(x),
\end{align*}
where $\chi\in C^\infty_c(\R^n)$ with $\chi=1$ on $B_{2R}(0)$, without changing the values of $(-\Delta)_\Omega^{s} u$ in $\Omega$. It is straightforward to check that $\tilde\rho$ satisfies \ref{ASS:RHO} and \ref{ASS:J} with $\alpha=2s$ and any $N>3-\alpha = 3-2s$. 
\end{remark}

\bigskip

\begin{remark} \normalfont
	In order to construct a nontrivial, \textit{anisotropic} density function $\rho$ with associated kernel $J$, such that \ref{ASS:RHO} and \ref{ASS:J} are fulfilled, we proceed as follows. 
	Let $\mathring \rho\in C^1\big(\R^n\setminus\{0\};[0,\infty)\big)$ be a radially symmetric function, which satisfies 
	\eqref{COND:RHO:1} and \eqref{COND:RHO:2}.
	Without loss of generality, we assume that $\norm{\mathring\rho}_{L^1(\R^n)} = 2n$. This implies that
	\begin{equation*}
		\frac{1}{2} \int_{\R^n}  \frac{\mathring\rho(z)}{\abs{z}^2} z\otimes z \dz = I .
	\end{equation*}
	Moreover, let $B\in \R^{n\times n}$ be a symmetric, positive definite matrix. 
	To simplify the computations, we further assume, without loss of generality, that $\det B=1$. We now define $\rho\in C^1\big(\R^n\setminus\{0\};[0,\infty)\big)$ with
	\begin{align*}
		\rho(x) \coloneqq \mathring\rho(Bx) \frac{\abs{x}^2}{\abs{Bx}^2}
		\quad\text{for all $x\in\R^n\setminus\{0\}$.}
	\end{align*}
	In this way, $\rho$ is clearly even but, in general, not any more radially symmetric.
	Its associated kernel can be expressed as 
	\begin{align*}
		J(x) =  \frac{\mathring\rho(Bx)}{\abs{Bx}^2}
		\quad\text{for all $x\in\R^n\setminus\{0\}$.}
	\end{align*}
	It is easy to check that $\rho$ and $J$ satisfy \ref{ASS:RHO} and Condition~\eqref{COND:RHO:2}. Moreover, a straightforward computation yields $M = B^{-2}$ and $A = B^{-1}$. To verify \eqref{COND:J:2}, let $Q\in \mathrm{SO}(n)$ be arbitrary. 
	Recalling that $\mathring\rho$ is radially symmetric, we deduce that
	\begin{align*}
		&\int_{\R^{n-1}} J\left(AQ
		\begin{pmatrix} z' \\ z_n \end{pmatrix}
		\right) z' \,\mathrm{d}z'
		= \int_{\R^{n-1}} \frac{\mathring\rho(BAQ z)}{| BAQ z |^2} z' \,\mathrm{d}z'
		\\
		&\quad
        = \int_{\R^{n-1}} \frac{\mathring\rho(Q z) }{| Q z |^2} z' \,\mathrm{d}z'
		= \int_{\R^{n-1}} \frac{\mathring\rho(z)}{|z|^2} z' \,\mathrm{d}z'
		= 0
	\end{align*}
    for all $z_n \in \R$.
	This shows that $J$ satisfies Condition~\eqref{COND:J:2} and consequently, \ref{ASS:J} is also fulfilled.
\end{remark}

\medskip

The following lemma provides two important properties of the family $\{\rho_\eps\}_{\eps>0}$.

\begin{lemma}\label{LEM:DIRAC}
    Suppose that Assumption~\ref{ASS:RHO} holds. 
    \begin{enumerate}[label=\textnormal{(\alph*)}, topsep=0ex, leftmargin=*]
        \item\label{LEM:DIRAC:A}
        For any $\eps>0$, it holds $\rho_\eps\in L^1(\R^n)$ with
        \begin{equation}
            \label{ID:RHOL1}
            \norm{\rho_\eps}_{L^1(\R^n)} = \norm{\rho}_{L^1(\R^n)} < \infty.
        \end{equation}
        \item\label{LEM:DIRAC:B} 
        For any $\delta>0$, it holds that
        \begin{equation}
            \underset{\eps\to 0}{\lim}\; \int_{\R^n\setminus B_\delta(0)} \rho_\eps(x) \dx = 0.
        \end{equation}
    \end{enumerate}
\end{lemma}

\begin{proof}
    Let $\eps>0$ be arbitrary. 
    Since $\rho \in L^1(\R^n)$ and the identity in \eqref{ID:RHOL1} follows from the change of variables $x\mapsto \tfrac x\eps$. Thus, assertion (a) is verified.

    Let now $\delta>0$ be arbitrary.
    The change of variables $x\mapsto \tfrac x\eps$ yields
    \begin{equation*}
        \int_{\R^n\setminus B_\delta(0)} \rho_\eps(x) 
        = \int_{\R^n\setminus B_{\delta/\eps}(0)} \rho(x) \dx.
    \end{equation*}
    Since $\rho\in L^1(\R^n)$, the left-hand side converges to zero as $\eps\to 0$ due to Lebesgue's dominated convergence theorem. This verifies (b).
\end{proof}




\section{Nonlocal-to-local convergence of the nonlocal operator} \label{SECT:MAIN}

In general, we assume that \ref{ASS:RHO} holds and
we consider a sufficiently smooth domain $\Omega \subseteq\R^n$.
The main goal of this paper is to investigate the convergence of the nonlocal operator $\LL_\eps^\Omega$, which is (at first formally) given by
\begin{equation}
    \label{DEF:LLE}
    \LL_\eps^\Omega u(x) = \mathrm{P.V.}\int_{\Omega} J_\eps(x-y) \big( u(x) - u(y) \big) \dy
\end{equation}
for all $x\in\Omega$ and any sufficiently regular function $u:\Omega\to\R$, 
to the local differential operator $\LL$,
which is defined as
\begin{align}
    \label{DEF:LL}
    \LL^\Omega u(x) = - \Div\big(M\nabla u(x)\big) = - \sum_{i,j=1}^n M_{ij}\, \partial_{x_j}\partial_{x_k}u(x)
\end{align}
for all $x\in\Omega$ and any sufficiently regular function $u:\Omega\to\R$.
Here, $M\in\R^{n\times n}$ is the symmetric matrix introduced in \ref{ASS:RHO}, which is positive definite.
If the boundary of $\Omega$ is not empty, the differential operator $\LL^\Omega$ is equipped with the natural boundary condition
\begin{equation}
	M \nabla u \cdot \n_{\partial\Omega} = 0 \quad\text{on $\partial\Omega$},
\end{equation}
where $\n_{\partial\Omega}$ denotes the a unit normal vector field on $\partial\Omega$. For any $k\ge 2$ and $p\in [1,\infty]$, we define the corresponding Sobolev spaces as
\begin{equation}
	\label{DEF:WKM}
	W^{k,p}_M(\Omega) = \big\{ u\in W^{k,p}(\Omega) \,:\, 	
	M \nabla u \cdot \n_{\partial\Omega} = 0 \;\text{on $\partial\Omega$}
	\big\}.
\end{equation}


\subsection{Well-definedness of the nonlocal operator} \label{SUBSEC:WELLDEF}

In this subsection, we first show that the expression
\begin{equation*}
    \LL_\eps^\Omega u(x) 
            = \mathrm{P.V.} \int_{\R^n} J_\eps(x-y) \big( u(x) - u(y) \big) \dy
\end{equation*}
is actually well defined if Assumption~\ref{ASS:RHO} is fulfilled. In case $\Omega = \R^n$, the following proposition shows that $\LL_\eps^\Omega$ actually defines a bounded linear operator.

\begin{proposition}\label{PRP:WP:R^N}
    Let $\Omega=\R^n$, suppose that Assumption~\ref{ASS:RHO} holds and let $p\in [1,\infty)$ and $\eps>0$ be arbitrary. Then, the following statements hold.
    \begin{enumerate}[label = \textnormal{(\alph*)}, topsep=0ex, leftmargin=*]
    \item\label{PRP:WP:R^N:A} The operator
    \begin{equation}
    \label{DEF:LRN:CC2}
        \begin{split}
            &\LL_\eps^\Omega: C_c^2(\R^n) \to L^p(\R^n),
            \\
            &\LL_\eps^\Omega u(x) = \int_{\R^n} J_\eps(x-y) \big( u(x) - u(y) - \nabla u(y)\cdot(x-y) \big) \dy
        \end{split}
    \end{equation}
    is well-defined and linear. Moreover, for any $u\in C_c^2(\R^n)$, $\LL_\eps^\Omega u$ can be expressed as
    \begin{equation}
    \label{REP:LE:PV}
        \LL_\eps^\Omega u(x) = \mathrm{P.V.}\int_{\R^n} J_\eps(x-y) \big( u(x) - u(y) \big) \dy
    \end{equation}
    for almost all $x\in\R^n$.
    \item\label{PRP:WP:R^N:B} The operator introduced in \textnormal{(a)} can be extended to a bounded linear operator
    \begin{equation*}
        \LL_\eps^\Omega: W^{2,p}(\R^n) \to L^p(\R^n)
        \quad\text{with}\quad
        \norm{\LL_\eps^\Omega u}_{L^p(\R^n)} \leq C_* \|u\|_{W^{2,p}(\R^n)},
    \end{equation*}
    for all $u\in W^{2,p}(\R^n)$, where the constant $C_*>0$ depends only on $\rho$ and $p$.
    \end{enumerate}
\end{proposition}

\medskip

By means of Proposition~\ref{PRP:WP:R^N}, we can draw conclusions on the well-definedness of the nonlocal operator $\LL_\eps^\Omega$ if $\Omega$ is a domain of class $C^2$.

\begin{corollary}\label{COR:WD:DOM}
	Let $\Omega$ be a (not necessarily bounded) domain with $C^2$-boundary, suppose that \ref{ASS:RHO} holds, and let $p\in [1,\infty)$ and $\eps>0$ be arbitrary. 
	Then, the operator
    \begin{align}
    \label{DEF:LOM:CC2L1LOC}
        \begin{split}
            &\LL_\eps^\Omega: C_c^2(\overline\Omega) \to L^p_{\mathrm{loc}}(\Omega),
            \\
            &\LL_\eps^\Omega u(x) 
            = \mathrm{P.V.} \int_{\Omega} J_\eps(x-y) \big( u(x) - u(y) \big) \dy
        \end{split}
    \end{align}
    is well-defined and linear. 
\end{corollary}

Provided that \ref{ASS:RHO} and \ref{ASS:J} and certain assumptions on the domain are fulfilled, we will see later that the operator introduced in \eqref{DEF:LOM:CC2L1LOC} can actually be extended to a linear and bounded operator mapping from $W^{2,p}(\Omega)$ to $L^p(\Omega)$.

In the remainder of this subsection, we present the proofs of Proposition~\ref{PRP:WP:R^N} and Corollary~\ref{COR:WD:DOM}.

\begin{proof}[Proof of Proposition~\ref{PRP:WP:R^N}]
    Let $p\in [1,\infty)$ and $\eps>0$ be arbitrary. 
    In the following, the letter $C$ will denote generic positive constants depending only on $\rho$ and $p$. 
    The concrete value of $C$ may vary throughout the proof.
    
    \noindent\textbf{Proof of \ref{PRP:WP:R^N:A}}. Let $u\in C^2_c(\R^n)$ be arbitrary. 
    By means of Taylor's theorem, we have
    \begin{align}
    \label{EQ:TAYLOR:1}
        u(x) - u(y) - \nabla u(x)\cdot(x-y) = -R_2(x,y),
    \end{align}
    where the error term is given by
    \begin{align}
    \label{EQ:TAYLOR:2}
        R_2(x,y) \coloneqq \sum_{|\beta|=2}\frac{2}{\beta!}\int_0^1(1-t)D^\beta u\big( y+t(x-y) \big)\dt\: (x-y)^\beta. 
    \end{align}
    Therefore, we obtain
    \begin{align*}
        &\int_{\R^n} \big|J_\eps(x-y)  \big( u(x) - u(y) - \nabla u(x)\cdot(x-y) \big) \big| \dy 
        \\
        &\quad\leq \sum_{|\beta|=2}\frac{2}{\beta!}\int_{\R^n}\int_0^1 
            \big| J_\eps(x-y)\,(1-t)\,D^\beta u\big( y+t(x-y) \big)\: (x-y)^\beta \big| \dt\dy 
        \\
        &\quad\leq C\,\|u\|_{W^{2,\infty}(\R^n)}\int_{\R^n}|\rho_\eps(x-y)|\dy
        = C\,\|u\|_{W^{2,\infty}(\R^n)}\,\norm{\rho}_{L^1(\R^n)}.
    \end{align*}
    This proves that for almost all $x\in\R^n$,
    \begin{equation}
        \label{PROP:L1}
        J_\eps(x-\cdot) \big(u(x) - u(\cdot) - \nabla u(x)\cdot(x-\cdot) \big) \in L^1(\R^n)
    \end{equation}
    and therefore, the integral in the definition of $\LL_\eps^\Omega$ actually exists.    
    Recalling \eqref{EQ:TAYLOR:1} and \eqref{EQ:TAYLOR:2}, we use Hölder's inequality and Fubini's theorem to deduce
    \begin{align*}
        &\|\LL_\eps^\Omega u\|_{L^p(\R^n)}^p 
        \\
        &\quad= \int_{\R^n}\Bigg|\sum_{|\beta|=2}
        \frac{2}{\beta!}\int_{\R^n}\int_0^1J_\eps(x-y)\,(1-t)\,D^\beta u\big( y+t(x-y) \big)\: (x-y)^\beta\dt\dy\, \Bigg|^p
        \!\!\!\dx 
        \\
        &\quad\leq C\int_{\R^n} \sum_{|\beta|=2} 
        \left(\int_{\R^n}\int_0^1|\rho_\eps(x-y)||D^\beta u\big( y+t(x-y) \big)|\dt\dy\right)^p\dx 
        \\
        &\quad\leq C\int_{\R^n} \norm{\rho}_{L^1(\R^n)}^{\frac{p-1}{p}} \sum_{|\beta|=2} 
        \int_{\R^n}\int_0^1|\rho_\eps(x-y)||D^\beta u\big( y+t(x-y) \big)|^p\dt\dy\dx
        \\
        &\quad\leq C \sum_{|\beta|=2} \int_0^1\int_{\R^n} 
        \int_{\R^n}|\rho_\eps(x-y)||D^\beta u\big( y+t(x-y) \big)|^p\dy\dx\dt.
    \end{align*}
    Applying first the change of variables $y\mapsto z= x-y$ and afterwards the change of variables $x\mapsto w= x-z+tz$, we infer
    \begin{align}
    \label{EST:LE:LP}
        \|\LL_\eps^\Omega u\|_{L^p(\R^n)}^p 
        &\leq C \sum_{|\beta|=2} \int_0^1
        \left(\int_{\R^n}|\rho_\eps(z)|\dz\right)\left(\int_{\R^n}|D^\beta u(w)|^p\dw\right) \dt
        \notag\\
        &\leq C \norm{\rho}_{L^1(\R^n)} \sum_{|\beta|=2} \left(\int_{\R^n}|D^\beta u(w)|^p\dw\right) 
        \leq C_*^p \, \|u\|_{W^{2,p}(\R^n)}^p
    \end{align}
    for some constant $C_*>0$ depending only on $\rho$ and $p$.
    This proves that the operator $\LL_\eps^\Omega$ is well-defined  and bounded in the sense of \eqref{EST:LE:LP}. Moreover, it is easy to check that the operator $\LL_\eps^\Omega$ is linear.

    It remains to verify the representation \eqref{REP:LE:PV}. To this end, let $r>0$ be arbitrary. Recalling \eqref{PROP:L1}, we deduce
    \begin{align}
        \label{EST:LE:DEL:1}
        &\int_{\abs{x-y}\ge r} 
            J_\eps(x-y) \big(u(x) - u(y)\big)  \dy 
        \notag\\
        &= \int_{\abs{x-y}\ge r} 
            J_\eps(x-y) \big(u(x) - u(y) - \nabla u(x)\cdot(x-y)\big)  \dy
        \notag\\
        &\quad+ \int_{\abs{x-y}\ge r} 
            J_\eps(x-y) \nabla u(x)\cdot(x-y) \dy
    \end{align}
    for almost all $x\in\R^n$.
    We point out that
    \begin{align*}
        &\int_{\abs{x-y}\ge r} 
            \big| J_\eps(x-y) \nabla u(x)\cdot(x-y) \big| \dy
        \\
        &\quad\leq \int_{\abs{x-y}\ge r} 
            \abs{\rho_\eps(x-y)} \abs{\nabla u(x)} \frac{1}{\abs{x-y}} \dy
        \leq \frac{1}{r} \norm{u}_{W^{1,\infty}(\R^n)} \norm{\rho}_{L^1(\R^n)}.
    \end{align*}
    This means that the second integral on the right-hand side of \eqref{EST:LE:DEL:1} actually exists and therefore, the identity \eqref{EST:LE:DEL:1} is justified. 
    As $\rho$ is an even function so is the kernel $J_\eps$.       
    This implies that
    \begin{align}
        \label{INT:LE:1}
        \int_{\abs{x-y}\ge r} J_\eps(x-y) \nabla u(x)\cdot(x-y) \dy = 0.
    \end{align}
    Moreover, invoking \eqref{PROP:L1}, we obtain
    \begin{equation}
        \label{INT:LE:2}
        \int_{\abs{x-y}\ge r} 
            J_\eps(x-y) \big(u(x) - u(y) - \nabla u(x)\cdot(x-y)\big)  \dy
        \to \LL_\eps^\Omega u(x)
    \end{equation}
    as $r \to 0$ for almost all $x\in\R^n$ by means of Lebesgue's dominated convergence theorem.
    Eventually, combining \eqref{EST:LE:DEL:1}, \eqref{INT:LE:1} and \eqref{INT:LE:2}, we conclude that
    \begin{equation*}
        \LL_\eps^\Omega u(x) 
        = \underset{r\searrow 0}{\lim} \int_{\abs{x-y}\ge r} 
            J_\eps(x-y) \big(u(x) - u(y)\big)  \dy .
    \end{equation*}
    By the definition of the principal value, this is exactly \eqref{REP:LE:PV}.
    
    \noindent\textbf{Proof of \ref{PRP:WP:R^N:B}}. 
    We already know from \ref{PRP:WP:R^N:A} that the operator $\LL_\eps^\Omega: C_c^2(\R^n) \to L^p(\R^n)$ is well-defined, linear and bounded in the sense of \eqref{EST:LE:LP}.
    As $C_c^2(\R^n)$ is dense in $W^{2,p}(\R^n)$, 
    we conclude via estimate \eqref{EST:LE:LP} that $\LL_\eps^\Omega$ can be extended to a bounded linear operator
    \begin{align*}
        \LL_\eps^\Omega: W^{2,p}(\R^n) \to L^p(\R^n)
        \quad\text{with}\quad
        \norm{\LL_\eps^\Omega u}_{L^p(\R^n)} \leq C_* \|u\|_{W^{2,p}(\R^n)}.
    \end{align*}
    Hence, the proof is complete.    
\end{proof}

\medskip

\begin{proof}[Proof of Corollary~\ref{COR:WD:DOM}]
Without loss of generality, we assume that $\Omega\subsetneq \R^n$ as the case $\Omega=\R^n$ was already handled in Proposition~\ref{PRP:WP:R^N}. 
Let $u\in C^2_c(\overline\Omega)$ be arbitrary. Since $\Omega$ is of class $C^2$, we can find an extension 
$\tilde{u}\in C^2_c(\R^n)$ with $\tilde{u}\vert_{\Omega} = u$. 
We now fix an arbitrary $x\in \Omega$. Since $\Omega \subset \R^n$ is open, we can find a radius $r(x)>0$ such that $B_{r(x)}(x) \subset \Omega$. 
Then, for any $r\in (0,r(x)]$, we have $B_{r}(x) \subset \Omega$ and thus,
\begin{align}
    \label{DECOMP:OMBR}
    \begin{split}
    &\int_{\Omega\setminus B_{r}(x)} J_\eps (x-y) \big( u(x) - u(y) \big) \dy
    \\
    &= \int_{\R^n\setminus B_r(x)} J_\eps (x-y) \big( \tilde{u}(x) - \tilde{u}(y) \big) \dy
    - \int_{\R^n\setminus \Omega} J_\eps (x-y) \big( \tilde{u}(x) - \tilde{u}(y) \big) \dy
    \end{split}
\end{align}
We already know from Proposition~\ref{PRP:WP:R^N}\ref{PRP:WP:R^N:A} that
\begin{equation*}
    \int_{\R^n\setminus B_r(x)} J_\eps (x-y) \big( \tilde{u}(x) - \tilde{u}(y) \big) \dy \to \LL_\eps^{\R^n} \tilde{u} (x)
    \quad\text{as $r\to 0$.}
\end{equation*}
Therefore, it remains to show that the second integral on the right-hand side of \eqref{DECOMP:OMBR} actually exists. Since $B_{r(x)}(x) \subset \Omega$, we obviously have 
\begin{equation*}
    |x-y| \ge r(x) 
    \quad\text{for all $y\in \R^n\setminus \Omega$}.
\end{equation*}
Recalling the definition of $J_\eps$, this implies that
\begin{align*}
    &\left| \int_{\R^n\setminus \Omega} J_\eps (x-y) 
        \big( \tilde{u}(x) - \tilde{u}(y) \big) \dy \right|
    \le \int_{\R^n\setminus \Omega} J_\eps(x-y) 
        \big|\tilde{u}(x) - \tilde{u}(y) \big| \dy 
    \\
    &\quad\le \frac{2}{r(x)^2}\, \norm{\tilde{u}}_{L^\infty(\R^n)} \norm{\rho_\eps}_{L^1(\R^n)}
    = \frac{2}{r(x)^2}\, \norm{\tilde{u}}_{L^\infty(\R^n)} \norm{\rho}_{L^1(\R^n)}
    < \infty. 
\end{align*}
Hence, sending $r\to 0$ in \eqref{DECOMP:OMBR}, we conclude that the expression
\begin{align}
    \label{WD:PV}
    \begin{split}
    &\LL_\eps^\Omega u(x) 
    = \mathrm{P.V.}\int_{\Omega} J_\eps(x-y) \big( u(x) - u(y) \big) \dy
    \\
    &\quad  
    = \underset{r\searrow 0}{\lim}\,\int_{\Omega\cap\{\abs{x-y}\ge r\}} 
        J_\eps (x-y) \big( u(x) - u(y) \big) \dy
    \\
    &\quad
    = \LL_\eps^{\R^n} \tilde{u} (x) 
        - \int_{\R^n\setminus \Omega} J_\eps (x-y) \big( \tilde{u}(x) - \tilde{u}(y) \big) \dy
    \end{split}
\end{align}
is well-defined for every $x\in \Omega$ provided that $u\in C^2_c(\overline\Omega)$.

Let now $K$ be an arbitrary compact subset of $\Omega$. Since $\Omega$ is open, we know that $r(K) \coloneqq \dist(\R^n\setminus\Omega,K) > 0$. As the pointwise limit of a sequence of measurable functions is measurable, we further know that the mapping 
\begin{equation*}
    \Omega \ni x \mapsto \LL_\eps^\Omega u(x) \in \R
\end{equation*}
is measurable. Using \eqref{WD:PV} along with Proposition~\ref{PRP:WP:R^N}, and proceeding similarly as above, we infer that
\begin{align*}
    \norm{\LL_\eps^\Omega u}_{L^p(K)}^p
    &\le C \norm{\LL_\eps^{\R^n} \tilde{u}}_{L^p(\R^n)}^p
    + \int_K \Bigg| \int_{\R^n\setminus \Omega} J_\eps (x-y) \big( \tilde{u}(x) - \tilde{u}(y) \big) \dy \Bigg|^p \dx
    \\
    &\le C \norm{\LL_\eps^{\R^n} \tilde{u}}_{L^p(\R^n)}^p
    + \frac{2^p |K|}{r(K)^{2p}}\, \norm{\tilde{u}}_{L^\infty(\R^n)}^p \, \norm{\rho}_{L^1(\R^n)}^p < \infty.
\end{align*}
As the compact subset $K$ was arbitrary, this proves that $\LL_\eps^\Omega u \in L^p_\mathrm{loc}(\Omega)$. This means that the operator
\begin{equation*}
    \LL_\eps^\Omega: C_c^2(\overline\Omega) \to L^p_{\mathrm{loc}}(\Omega),
    \quad
    u \mapsto \LL_\eps^\Omega u
\end{equation*}
is well-defined. Moreover, the operator is obviously linear and thus, the proof is complete.
\end{proof}


\subsection{Convergence on the full space} \label{SUBSEC:FULLSPACE}

In this subsection, we consider the case where $\Omega$ is the full space, i.e., $\Omega=\R^n$.
Based on the results established in Proposition~\ref{PRP:WP:R^N}, we obtain the following nonlocal-to-local convergence properties.

\begin{theorem}\label{THM:CONV:R^N}
	We consider $\Omega=\R^n$.
    Suppose that Assumption~\ref{ASS:RHO} holds and let $p\in [1,\infty)$ be arbitrary. 
    Then, the following statements hold.
    \begin{enumerate}[label = \textnormal{(\alph*)}, topsep=0ex, leftmargin=*]
        \item\label{THM:CONV:R^N:A} For all $u\in W^{2,p}(\R^n)$, 
        \begin{equation*}
            \LL_\eps^\Omega u \to \LL^\Omega u 
            \quad\text{in $L^p(\R^n)$ as $\eps\to 0$.}
        \end{equation*}

        \item\label{THM:CONV:R^N:B} There exists a constant $C>0$ such that for all $\eps>0$ and all $u\in W^{3,p}(\R^n)$, it holds
        \begin{equation*}
            \|\LL_\eps^\Omega u - \LL^\Omega u\|_{L^p(\R^n)} \leq C\eps\|u\|_{W^{3,p}(\R^n)}.
        \end{equation*}
    \end{enumerate}
\end{theorem}

\medskip

\begin{proof}[Proof of Theorem~\ref{THM:CONV:R^N}] Let $p\in [1,\infty)$ and $\eps >0$ be arbitrary. 
In the following, the letter $C$ will denote generic positive constants depending only on $\rho$ and $p$. The concrete value of $C$ may vary throughout the proof.

\noindent\textbf{Proof of \ref{THM:CONV:R^N:B}}. 
    Since $C^\infty_c(\R^n)$ is dense in $W^{3,p}(\R^n)$, we first verify the assertion for functions $u\in C^\infty_c(\R^n)$. Therefore, let $u\in C^\infty_c(\R^n)$ be arbitrary.
    By the definitions of $\LL_\eps^\Omega$ and $\LL$ (see~\eqref{DEF:LRN:CC2} and \eqref{DEF:LL}), we obtain
    \begin{align}
        \label{Splitting}
        &\|\LL_\eps^\Omega u - \LL^\Omega u\|_{L^p(\R^n)}^p 
        \nonumber \\
        &\quad= \int_{\R^n}\Bigg|\int_{\R^n}J_\eps(x-y)
            \big( u(x)-u(y) + \nabla u(x)\cdot (x-y) \big)\dy 
            + M:D^2u(x)\Bigg|^p\dx 
        \nonumber \\
        &\quad\leq  C\big( I_{1,\eps} + I_{2,\eps} \big), \phantom{\int_{\R^n}}
    \end{align}
    where
    \begin{align*}
        I_{1,\eps} &\coloneqq 
            \int_{\R^n}\left|\int_{B_1(x)} \!\! J_\eps(x-y)
            \big( u(x)-u(y) + \nabla u(x)\cdot (x-y) \big)\dy 
            \,+\, M:D^2u(x)\right|^p \!\!\!\dx
        \\
        I_{2,\eps} &\coloneqq
            \int_{\R^n}\left|\int_{\R^n\setminus B_1(x)}J_\eps(x-y)
            \big( u(x)-u(y) + \nabla u(x)\cdot (x-y) \big)\dy \right|^p\dx.
    \end{align*}
    We now estimate the terms $I_{1,\eps}$ and $I_{2,\eps}$ separately. 
    
    \noindent\textit{Ad $I_{1,\eps}$}: Applying Taylor's theorem, we obtain the expansion
    \begin{align}
    \label{EQ:TAYLOR:3}
        u(x) - u(y) + \nabla u(x)\cdot (x-y) = - \frac{1}{2}(x-y)^TD^2u(x)(x-y) - R_3(x,y),
    \end{align}
    where the error term is given by
    \begin{align}
        \label{EQ:TAYLOR:4}
        R_3(x,y) = \sum_{|\beta|=3}\frac{3}{\beta!}\int_0^1(1-t)\,D^\beta u\big( y+t(x-y) \big)\dt\; (x-y)^\beta.
    \end{align}
    Therefore, $I_{1,\eps}$ can be estimated as
    \begin{align}
    \label{EST:I1E}
        I_{1,\eps} 
        &\le C \int_{\R^n}\Bigg| -\int_{B_1(x)} \frac{1}{2} J_\eps(x-y)
            (x-y)^TD^2u(x)(x-y) \dy \,+\, M:D^2u(x)\Bigg|^p \dx
        \notag\\
        &\quad + C \int_{\R^n} \Bigg| \int_{B_1(x)}J_\eps(x-y) R_3(x,y) \dy \Bigg|^p \dx.
    \end{align}
    Employing the change of variables $x\mapsto z=x-y$ and recalling the definition of $M$ (see~\eqref{DEF:M}), we observe that 
    \begin{align}
        \label{Calc:FS:1}
        &- \int_{B_1(x)} \frac{1}{2} J_\eps(x-y)(x-y)^TD^2u(x)(x-y) \dy 
            \,+\, M:D^2u(x) \dx
        \nonumber\\
        &\quad= - \frac{1}{2} \int_{B_1(0)} J_\eps(z)z^T D^2u(x) z\dz 
            \,-\,  \frac{1}{2} \int_{\R^n} J_\eps(z)z^T D^2u(x) z\dz
        \nonumber\\
        &\quad=   \int_{\R^n\setminus B_1(0)}J_\eps(z)z^TD^2u(x)z\dz. 
    \end{align}
    for all $x\in\R^n$. Plugging this identity into \eqref{EST:I1E}, we obtain 
    \begin{align}
    \label{EST:I1E:2}
        I_{1,\eps} 
        &\le C \int_{\R^n}\Bigg| \int_{\R^n\setminus B_1(0)}J_\eps(z)z^TD^2u(x)z\dz \Bigg|^p\dx
        \notag\\
        &\quad + C \int_{\R^n} \Bigg| \int_{B_1(x)}J_\eps(x-y) R_3(x,y) \dy \Bigg|^p \dx.
    \end{align}
    Since $|z|\geq 1$ for all $z\in \R^n\setminus B_1(0)$, the first summand on the right-hand side of \eqref{EST:I1E:2} can be estimated as
    \begin{align*}
        &\int_{\R^n}\left| \int_{\R^n\setminus B_1(0)}J_\eps(z)z^TD^2u(x)z\dz\right|^p\dx 
        \\
        &\quad\leq \int_{\R^n}\left(\int_{\R^n\setminus B_1(0)} \abs{\rho_\eps(z)}\,|z|\,\abs{D^2u(x)} \dz\right)^p\dx 
        \\
        &\quad\leq C\int_{\R^n}\left(\int_{\R^n} \abs{\rho_\eps(z)}\,|z|\dz\right)^p \abs{D^2u(x)}^p \dx 
        \leq C\eps^p\|u\|_{W^{2,p}(\R^n)}^p.
    \end{align*}
    Recalling the definition of the error term, we obtain the following estimate for the the second summand on the right-hand side of \eqref{EST:I1E:2}: 
    \begin{align*}
        &\int_{\R^n}\left|\int_{B_1(x)}J_\eps(x-y)R_3(x,y)\dy\right|^p\dx 
        \\
        &= \int_{\R^n}\left|\sum_{|\beta|=3}\frac{3}{\beta!}\int_{B_1(x)}\int_0^1 
            J_\eps(x-y)(1-t)D^\beta u\big( y+t(x-y) \big)\; (x-y)^\beta\dt\dy\right|^p\dx 
        \\
        &\leq C \sum_{|\beta|=3} \int_{\R^n}\left(\int_{B_1(x)}\int_0^1 \abs{\rho_\eps(x-y)}\, |x-y|\, 
            |D^\beta u\big( y+t(x-y) \big)|\dt\dy\right)^p\dx
    \end{align*}
    Next, we use the splitting 
    \begin{align}
        \label{SPLIT:PQ*}
        \abs{\rho_\eps(x-y)}\,|x-y| 
        = \big(\abs{\rho_\eps(x-y)}\,|x-y|\big)^{\frac{1}{q}} 
            \big(\abs{\rho_\eps(x-y)}\, |x-y|\big)^{\frac{1}{p}},
    \end{align}
    where $q$ denotes the conjugate exponent to $p$, i.e., $\frac{1}{q}+\frac{1}{p} =1$. Then, using Hölder's inequality as well as the changes of variables $y\mapsto z= x-y$ and $x\mapsto w= x-z+tz$, we infer 
    \begin{align*}
        &\int_{\R^n}\left|\int_{B_1(x)}J_\eps(x-y)R_3(x,y)\dy\right|^p\dx 
        \\
        &\quad
        \leq C\sum_{|\beta|=3}\int_{\R^n} \left(\int_{\R^n} \abs{\rho_\eps(z)}\,|z|\dz\right)^{p-1}
        \\[-0.5ex]
        &\qquad\qquad\qquad\quad
            \cdot \left(\int_{\R^n}\int_0^1 \abs{\rho_\eps(z)}\,|z|\, 
            \abs{D^\beta u(x-z + tz)}^p \dt\dz \right)\dx 
        \\[1ex]
        &\quad\leq C\eps^{p-1} \sum_{|\beta|=3} \int_0^1 \int_{\R^n}\int_{\R^n} \abs{\rho_\eps(z)}\,|z|\, 
            \abs{D^\beta u(w)}^p \dz\dw \dt
        \\[1.5ex]
        &\quad\leq C\eps^p\|u\|_{W^{3,p}(\R^n)}^p.
    \end{align*} 
    Altogether, this shows that
    \begin{align}
        \label{EST:I1E:FIN}
        I_{1,\eps} \leq C\eps^p \|u\|_{W^{3,p}(\R^n)}^p.
    \end{align}
    
    \noindent\textit{Ad $I_{2,\eps}$}: Applying Taylor's theorem, we obtain the expansion
    \begin{align}
    \label{EQ:TAYLOR:5}
        u(x) - u(y) - \nabla u(x)\cdot(x-y) = -R_2(x,y),
    \end{align}
    where the error term is given by
    \begin{align}
    \label{EQ:TAYLOR:6}
        R_2(x,y) \coloneqq \sum_{|\beta|=2}\frac{2}{\beta!}\int_0^1(1-t)D^\beta u\big( y+t(x-y) \big)\dt\: (x-y)^\beta. 
    \end{align}
    Hence, $I_{2,\eps}$ can be estimated as
    \begin{align*}
        I_{2,\eps}
        &= C\int_{\R^n}\left|\int_{\R^n\setminus B_1(x)}J_\eps(x-y)
            R_2(x,y) \dy \right|^p\dx
        \nonumber\\
        &\le C \sum_{|\beta|=2} \int_{\R^n}\left( \int_{\R^n\setminus B_1(x)} \int_0^1 
            \abs{\rho_\eps(x-y)} \abs{D^\beta u\big( y+t(x-y) \big)} \dy \dt \right)^p\dx
    \end{align*}
    Using the splitting
    \begin{align}
        \label{SPLIT:PQ}
        \abs{\rho_\eps(x-y)}
        = \abs{\rho_\eps(x-y)}^{\frac{1}{q}} 
            \abs{\rho_\eps(x-y)}^{\frac{1}{p}},
    \end{align}
    where $\frac{1}{q}+\frac{1}{p} =1$,
    Hölder's inequality, and the changes of variables $y\mapsto z= x-y$ and $x\mapsto w = x-z+tz$, we infer 
    \allowdisplaybreaks
    \begin{align}
        \label{EST:I2E:FIN}
        I_{2,\eps}
        &\leq C\sum_{|\beta|=2}\int_{\R^n}\left(\int_{\R^n} \abs{\rho_\eps(z)}\dz\right)^{p-1}
        \nonumber\\[-0.5ex]
        &\qquad\qquad\qquad
            \cdot \left(\int_{\R^n\setminus B_1(0)}\int_0^1 \abs{\rho_\eps(z)}\, 
            \abs{D^\beta u(x-z + tz)}^p \dt\dz \right)\dx 
        \nonumber\\[1ex]
        &\leq C\eps^{p-1} \sum_{|\beta|=2} \int_0^1 \int_{\R^n}\int_{\R^n\setminus B_1(0)} 
            \abs{\rho_\eps(z)}\,|z|\, 
            \abs{D^\beta u(w)}^p \dz\dw \dt
        \nonumber\\[1.5ex]
        &\leq C\eps^p\|u\|_{W^{2,p}(\R^n)}^p,
    \end{align}
    \allowdisplaybreaks[0]
    
    Combining \eqref{Splitting} with \eqref{EST:I1E:FIN} and \eqref{EST:I2E:FIN}, we eventually conclude that
    \begin{align}
        \label{CONV:CCINF}
        \|\LL_\eps^\Omega u - \LL^\Omega u\|_{L^p(\R^n)} \leq C\eps\|u\|_{W^{3,p}(\R^n)}
    \end{align}
    for all $u\in C^\infty_c(\R^n)$.

    Let now $u\in W^{3,p}(\R^n)$ be arbitrary. Then, because of density, there exists a sequence $(u_k)_{k\in\N}\subseteq C^\infty_c(\R^n)$ such that $u_k \rightarrow u$ in $W^{3,p}(\R^n)$ as $k\rightarrow\infty$.
    In particular, this entails that $(u_k)_{k\in\N}$ is bounded in $W^{3,p}(\R^n)$.
    
    \newpage\noindent
    Thus, it follows from \eqref{CONV:CCINF} that
    \begin{align*}
        \Big(\LL_\eps^\Omega u_k - \LL^\Omega u_k\Big)_{k\in\N} \subseteq L^p(\R^n)
    \end{align*}
    is bounded. Hence, according to the Banach--Alaoglu theorem, there exists a function $w\in L^p(\R^n)$ such that, up to subsequence extraction,
    \begin{align*}
        \LL_\eps^\Omega u_k - \LL^\Omega u_k \rightharpoonup w 
        \quad \text{in $L^p(\R^n)$ as $k\rightarrow\infty$.}
    \end{align*}    
    Because of Proposition~\ref{PRP:WP:R^N}, we know that
    \begin{equation*}
        \LL_\eps^\Omega u_k \to \LL_\eps^\Omega u  
        \quad \text{in $L^p(\R^n)$ as $k\rightarrow\infty$.}
    \end{equation*}
    Moreover, recalling the definition of $\LL$, it is further clear that
    \begin{equation*}
        \LL^\Omega u_k \to \LL^\Omega u  
        \quad \text{in $L^p(\R^n)$ as $k\rightarrow\infty$.}
    \end{equation*}
    Consequently, due to the uniqueness of the weak limit, we have $w = \LL_\eps^\Omega u - \LL^\Omega u$.
    Thus, by means of weak lower semi-continuity, we conclude
    \begin{align*}
        &\|\LL_\eps^\Omega u - \LL^\Omega u\|_{L^p(\R^n)} 
        \leq \liminf_{k\rightarrow\infty}\|\LL_\eps^\Omega u_k - \LL^\Omega u_k\|_{L^p(\R^n)}
        \\
        &\quad \leq \limsup_{k\rightarrow\infty}C\eps\|u_k\|_{W^{3,p}(\R^n)} 
        = C\eps \|u\|_{W^{3,p}(\R^n)}.
    \end{align*}
    Since $u\in W^{3,p}(\R^n)$ was arbitrary, this proves (b).

    \medskip

    \noindent\textbf{Proof of \ref{THM:CONV:R^N:A}}.
    We already know from Proposition~\ref{PRP:WP:R^N}\ref{PRP:WP:R^N:B} that the operator norm of the nonlocal operator $\LL_\eps^\Omega: W^{2,p}(\R^n) \to L^p(\R^n)$ is bounded by the constant $C_*$, which is independent of $\eps$. Together with the definition of $\LL^\Omega u$ (see~\eqref{DEF:LL}), this implies that
    \begin{equation*}
        \|\LL_\eps^\Omega u - \LL^\Omega u\|_{L^p(\R^n)}
        \le \|\LL_\eps^\Omega u\|_{L^p(\R^n)}
            + \|\LL^\Omega u\|_{L^p(\R^n)}
        \le C\|u\|_{W^{2,p}(\R^n)}
    \end{equation*}
    for all $u\in W^{2,p}(\R^n)$.
    We further know from part \ref{THM:CONV:R^N:B} that for any $u\in W^{3,p}(\R^n)$, it holds that
    \begin{equation*}
        \LL_\eps^\Omega u - \LL^\Omega u \to 0
        \quad\text{in $L^p(\R^n)$ as $\eps\to 0$.}
    \end{equation*}
    As the inclusion $W^{3,p}(\R^n) \subset W^{2,p}(\R^n)$ is dense, the assertion of \ref{THM:CONV:R^N:A} directly follows from the Banach--Steinhaus theorem. 
    
    Hence, the proof of Theorem~\ref{THM:CONV:R^N} is complete.
\end{proof}


\subsection{Convergence on a curved half space} \label{SUBSEC:HALFSPACE}

In this section, we next consider the case where our domain is a curved half space. More precisely, we consider 
\begin{equation}
	\R^n_\gamma = \big\{ x\in\R^n \,:\, x_n > \gamma(x_1,\ldots,x_{n-1}) \big\}
\end{equation}
for a prescribed function $\gamma\in C^{3}_b(\R^{n-1})$. 

\medskip

\setstretch{1.15}
\begin{theorem}\label{THM:CONV:HS}
	Let $\gamma\in C^{3}_b(\R^{n-1})$, 
    $p\in [1,\infty)$, $\eps\in (0,1]$, and suppose that \ref{ASS:RHO} and \ref{ASS:J} hold with $M=I$. 
	Then, if $\norm{\gamma}_{C^1_b(\R^{n-1})}$ is sufficently small, 
	the operator introduced in \eqref{DEF:LOM:CC2L1LOC} $\big($restricted to $C^2_c(\overline{\R^n_\gamma}) \cap W^{2,p}_I(\R^n_\gamma)\big)$ can be extended to a bounded linear operator
    \begin{equation}
        \label{DEF:LOM:HS}
        \LL_\eps^{\R^n_\gamma}: W^{2,p}_I(\R^n_\gamma) \to L^p(\R^n_\gamma)
    \end{equation}
    with
    \begin{equation}
        \label{BND:LOM:HS}
        \norm{\LL_\eps^{\R^n_\gamma} u}_{L^p(\R^n_\gamma)} 
        \leq c_\gamma \|u\|_{W^{2,p}(\R^n_\gamma)}
    \end{equation}
    for all $u\in W^{2,p}_I(\R^n_\gamma)$, where $c_\gamma$ is a positive constant depending only on $\gamma$, $\rho$ and $p$.
    Moreover, there exists a positive constant $C_\gamma$ depending only on $\gamma$, $\rho$ and $p$ such that for all $\eps>0$ and all $u\in W^{3,p}_I(\R^n_\gamma)$, it holds
	\begin{equation}
		\label{CONV:CHS}
		\big\|\LL_\eps^{\R^n_\gamma} u + \Delta u \big\|_{L^p(\R^n_\gamma)} 
        \leq C_\gamma\, \sqrt[p]{\eps}\, \|u\|_{W^{3,p}(\R^n_\gamma)}.
	\end{equation}
    Furthermore, for all $u\in W^{2,p}_I(\R^n_\gamma)$, it holds
    \begin{equation}
		\label{CONV:CHS:2}
		\LL_\eps^{\R^n_\gamma} u \to - \Delta u
        \quad\text{in $L^p(\R^n_\gamma)$ as $\eps\to 0$.}
	\end{equation}
\end{theorem}

\medskip

\begin{corollary}\label{COR:CONV:HS}
	Let $\gamma\in C^{3}_b(\R^{n-1})$, $Q\in \mathrm{SO}(n)$, $\eps\in (0,1]$, and suppose that \ref{ASS:RHO} and \ref{ASS:J} hold with $M=I$. 
	Then, the results of Theorem~\ref{THM:CONV:HS} hold true for $Q\R^n_\gamma$ instead of $\R^n_\gamma$ and $W^{k,p}_Q(Q\R^n_\gamma)$ instead of $W^{k,p}_I(\R^n_\gamma)$ for $k=2,3$.
\end{corollary}
\setstretch{1.0}

\medskip

\begin{proof}[Proof of Theorem~\ref{THM:CONV:HS}]
    To provide a cleaner presentation, we will usually refrain from indicating the principal value by the symbol $\mathrm{P.V.}$ whenever the meaning is clear. In the following, the letter $C$ will denote generic positive constants depending only on $\gamma$, $\rho$ and $p$. The concrete value of $C$ may vary throughout the proof. The proof is split into three steps. 
    
    \noindent\textbf{Step~1: Proof of Estimate~\eqref{BND:LOM:HS} for $u\in C^2_c(\overline{\R^n_\gamma}) \cap W^{2,p}_I(\R^n_\gamma)$.}\\
    Let $u\in C^2_c(\overline{\R^n_\gamma}) \cap W^{2,p}_I(\R^n_\gamma)$ be arbitrary. Hence, according to Corollary~\ref{COR:WD:DOM}, the expression
    \begin{equation*}
        \LL_\eps^{\R^n_\gamma} u(x) = \mathrm{P.V.}\int_{\R^n_\gamma} J_\eps(x-y) \big( u(x) - u(y) \big) \dy
    \end{equation*}
    is well-defined.
    Our first goal is to show that $\LL_\eps^{\R^n_\gamma} u \in L^p(\R^n_\gamma)$ and that 
    \begin{equation}
        \label{EST:LOM:CC2*}
        \norm{\LL_\eps^{\R^n_\gamma} u}_{L^p(\R^n_\gamma)} \leq C \|u\|_{W^{2,p}(\R^n_\gamma)}.
    \end{equation}
    By construction, it is clear that $\LL_\eps^{\R^n_\gamma} u$ is measurable. Since $\R^n_\gamma$ is of class $C^3$, we can find an extension $\tilde{u} \in C^2_c(\R^n)$ with $u\vert_{\R^n_\gamma} = u$. In particular, $\tilde{u}$ can be chosen in such a way that
    \begin{equation}
        \label{EST:B:EXT}
        \norm{\tilde{u}}_{W^{2,p}(\R^n)} \le C \norm{u}_{W^{2,p}(\R^n_\gamma)}.
    \end{equation}    
    Using Proposition~\ref{PRP:WP:R^N}, we obtain the estimate
    \begin{align}\label{EST:B:1}
        \|\LL_\eps^{\R^n_\gamma} u\|_{L^p({\R^n_\gamma})}^p 
        &\leq C\|\LL_\eps^{\R^n}\tilde{u}\|_{L^p(\R^n)}^p + C\|\mathcal{R}_\eps\tilde{u}\|_{L^p({\R^n_\gamma})}^p \nonumber\\
        &\leq C\|u\|_{W^{2,p}({\R^n_\gamma})}^p + C\|\mathcal{R}_\eps\tilde{u}\|_{L^p({\R^n_\gamma})}^p,
    \end{align}    
    where the error term is given by 
	\begin{align*}
		\mathcal{R}_\eps \tilde{u}(x) \coloneqq \int_{({\R^n_\gamma})^c}J_\eps(x-y)\big( u(x)-\tilde{u}(y) \big) \dy\quad \text{ for a.e. }x\in\R^n_\gamma.
	\end{align*}
	Therefore, it remains to show that
	\begin{equation}
		 \label{EST:B:2}
		 \|\mathcal{R}_\eps\tilde{u}\|_{L^p({\R^n_\gamma})}^p 
		 \le C\|u\|_{W^{2,p}({\R^n_\gamma})}^p
	\end{equation}
    since \eqref{EST:LOM:CC2*} then follows by combining \eqref{EST:B:1} and \eqref{EST:B:2}.

	Since $\gamma\in C^{3}_b(\R^{n-1}) $, we infer from \cite[Lemma 2.1]{Schumacher2009} that there exists a $C^{2,1}$-diffeomorphism $F_\gamma:\R^n \to \R^n$ with $F_\gamma(\R^n_+) = \R^n_\gamma$, which satisfies
	\begin{equation}
		\label{COND:B:G:GAMMA}
		F_\gamma(x^\prime,0) = \begin{pmatrix} x^\prime \\ \gamma(x^\prime) \end{pmatrix}
		\quad\text{and}\quad
		\partial_{x_n}F_\gamma(x)|_{x_n=0} = -\mathbf{n}_{\partial{\R^n_\gamma}}(x^\prime,\gamma(x^\prime))
	\end{equation}
	for all $x'\in \R^{n-1}$.
    Since $F_\gamma \in C^{2,1}(\R^n;\R^n)$, we further have $DF_\gamma \in W^{2,\infty}(\R^n;\R^{n\times n})$ with
    \begin{equation}
        \label{EST:B:DFG}
        \norm{DF_\gamma}_{W^{2,\infty}}(\R^n;\R^{n\times n}) \le C \norm{\gamma}_{C^3_b(\R^{n-1})} \le C.
    \end{equation}  
    In the following, we write
	\begin{equation*}
		d_{F_\gamma} \coloneqq |\det DF_\gamma|
	\end{equation*}
	as an abbreviation.
    In particular, due to \eqref{EST:B:DFG}, we have
    \begin{equation}
        \label{EST:B:DET}
        \norm{d_{F_\gamma}}_{L^\infty(\R^n)} 
        \leq C.
    \end{equation}
	Moreover, assuming $\norm{\gamma}_{C^1_b(\R^{n-1})}$ to be sufficiently small, we can ensure that
	\begin{align}
        \label{Cond:B:Diffeo:an*}
		\sup_{x\in\R^n}|D F_\gamma(x) - I| \leq \frac 16 .
	\end{align}
	Now, by the change of variables by $x \mapsto F_\gamma(x)$, we infer that
	\begin{align}
		\label{EST:B:NL:1}
		&\|\mathcal{R}_\eps\tilde{u}\|_{L^p({\R^n_\gamma})}^p 
        =\int_{\R^n_\gamma}\left|\int_{(\R^n_\gamma)^c}J_\eps(x-y) \big(u(x)-\tilde{u}(y)\big) \dy\right|^p\dx
        \\ \nonumber
		&\quad=  \int_{\R^n_+}\left|\int_{\R^n_-}J_\eps\big( F_\gamma(x)-F_\gamma(y) \big)
		\Big(u(F_\gamma(x))-\tilde{u}(F_\gamma(y))\Big)\, d_{F_\gamma}(y)\dy\right|^p d_{F_\gamma}(x)\dx,
	\end{align}
    To simplify the notation, we introduce the functions
	\begin{alignat}{2}
        \label{DEF:B:W}
        w&: \R^n \to \R^n, &&\quad w(x) = \tilde{u}\big(F_\gamma(x)\big),
        \\
		\label{DEF:B:MG}
		G_\gamma&:\R^n\times\R^n\to \R^{n\times n}, &&\quad G_\gamma(x,y) \coloneqq \int_{0}^1 DF_\gamma(y + t(x-y))\dt.
	\end{alignat}
	In particular, this means that $G_\gamma(x,x) = DF_\gamma(x)$ for all $x\in\R^n$.
    Using the chain rule along with \eqref{EST:B:EXT}, \eqref{EST:B:DFG} and \eqref{EST:B:DET}, we deduce
    \begin{equation}
        \label{EST:B:WU}
        \norm{w}_{W^{2,p}(\R^n)} 
        \le C \norm{\tilde{u}}_{W^{2,p}(\R^n)}
        \le C \norm{u}_{W^{2,p}({\R^n_\gamma})}.
    \end{equation}
    Moreover, by means of the fundamental theorem of calculus, we obtain
	\begin{align*}
		F_\gamma(x) - F_\gamma(y) 
		=  G_\gamma(x,y) \, (x-y).
	\end{align*}
    
    Let now $x, \tilde{x}, y, \tilde{y} \in \R^n$ be arbitrary. 
    Recalling that $G_\gamma(\tilde{x},\tilde{x}) = DF_\gamma(\tilde{x})$, we infer from \eqref{Cond:B:Diffeo:an*} that
    \begin{equation}
        \label{EST:B:GG:1}
        \big| \big[ G_\gamma(\tilde{x},\tilde{x}) - I \big]  (x-y) \big| \le \frac 16 |x-y|.
    \end{equation}
    This implies that
    \begin{equation}
        \label{EST:B:GG:2}
        \big| G_\gamma(\tilde{x},\tilde{x})(x-y) \big|
        \ge \Big| |x-y| - \big| \big[ G_\gamma(\tilde{x},\tilde{x}) - I \big]  (x-y) \big| \Big| \ge \frac 56 |x-y|.
    \end{equation}
    Moreover, recalling the definition of $G_\gamma$ in \eqref{DEF:B:MG} and invoking once more \eqref{Cond:B:Diffeo:an*}, we deduce that
    \begin{align}
        \label{EST:B:GG:3}
        \begin{split}
        &\big| \big[ G_\gamma(\tilde{x},\tilde{y}) - G_\gamma(\tilde{x},\tilde{x}) \big]  (x-y) \big|
        \\
        &\quad \le \Big[ \big| G_\gamma(\tilde{x},\tilde{y}) - I \big| + \big| G_\gamma(\tilde{x},\tilde{x}) - I \big| \Big] \, |x-y|
        \le \frac 13 |x-y|.
        \end{split}
    \end{align}
    Combining \eqref{EST:B:GG:2} and \eqref{EST:B:GG:3}, we conclude that
    \begin{align}
        \label{EST:B:GG:4}
        \begin{split}
        &\big| G_\gamma(\tilde{x},\tilde{y})(x-y) \big|
        \\
        &\quad \ge \Big| \big| G_\gamma(\tilde{x},\tilde{x})(x-y) \big| - \big| \big[ G_\gamma(\tilde{x},\tilde{y}) - G_\gamma(\tilde{x},\tilde{x}) \big]  (x-y) \big| \Big|
        \ge \frac 12 |x-y|
        \end{split}
    \end{align}
    for all $x, \tilde{x}, y, \tilde{y} \in \R^n$.
    Next, using \eqref{DEF:B:MG}, we can reformulate \eqref{EST:B:NL:1} as 
    \begin{align}
        \label{ID:B:GG}
        \begin{split}
        &\|\mathcal{R}_\eps\tilde{u}\|_{L^p({\R^n_\gamma})}^p
        \\
	    &\quad=  \int_{\R^n_+}\left|\int_{\R^n_-}J_\eps\big( G_\gamma(x,y) (x-y)\big)
		\big(w(x)-w(y)\big)\, d_{F_\gamma}(y)\dy\right|^p d_{F_\gamma}(x)\dx.
        \end{split}
    \end{align}
    Now, we introduce the sets
    \begin{align}
        \label{DEF:AB}
        \mathcal A \coloneqq \R^{n-1} \times (0,2) \subset \R^n_+
        \quad\text{and}\quad
        \mathcal B \coloneqq \R^{n-1}\times (-2,0) \subset \R^n_- \,.
    \end{align}
    Hence, from \eqref{ID:B:GG} we infer that
    \begin{align}
        \label{EST:B:RE}
        \|\mathcal{R}_\eps\tilde{u}\|_{L^p({\R^n_\gamma})}^p
        \le C\big(I_\eps^1 + I_\eps^2 + I_\eps^3\big),
    \end{align}
    where
    \allowdisplaybreaks
    \begin{align*}
        I_\eps^1 &\coloneqq \int_{\R^n_+\setminus\mathcal A}
        \left|\int_{\mathcal B}J_\eps\big( G_\gamma(x,y)(x-y) \big)(w(x)-w(y)) \, d_{F_\gamma}(y) \dy\right|^p
        d_{F_\gamma}(x)\dx\,,
        \\
        I_\eps^2 &\coloneqq\int_{\R^n_+}\left|\int_{\R^n_-\setminus\mathcal B }J_\eps\big( G_\gamma(x,y)(x-y) \big)(w(x)-w(y))\,  d_{F_\gamma}(y) \dy\right|^p
        d_{F_\gamma}(x)\dx\,,
        \\
        I_\eps^3 &\coloneqq \int_{\mathcal A}\left|\int_{\mathcal B }J_\eps\big( G_\gamma(x,y)(x-y) \big)(w(x)-w(y))\, d_{F_\gamma}(y) \dy\right|^p
        d_{F_\gamma}(x)\dx\,.
    \end{align*}
    \allowdisplaybreaks[0]
    These integral terms will now be handled separately.

    \noindent\textit{Ad $I_\eps^1$:}
    We first observe that
    \begin{equation*}
        I_\eps^1 \leq C(I_\eps^{1,1} + I_\eps^{1,2}),
    \end{equation*}
    where
    \begin{align*}
        I_\eps^{1,1} &\coloneqq \int_{\R^n_+\setminus\mathcal A}
        \left(\int_{\mathcal B}J_\eps\big( G_\gamma(x,y)(x-y) \big)|w(x)| \, d_{F_\gamma}(y) \dy\right)^p
        \dx, \\
        I_\eps^{1,2} &\coloneqq \int_{\R^n_+\setminus\mathcal A}
        \left(\int_{\mathcal B}J_\eps\big( G_\gamma(x,y)(x-y) \big)|w(y)| \, d_{F_\gamma}(y) \dy\right)^p
        \dx.
    \end{align*}
    We further observe that for all $x \in \R^n_+\setminus \mathcal A$ and all $y \in \mathcal B$, we have
    \begin{align*}
        \abs{ x - y} 
        \ge \abs{ x_n -  y_n} 
        =  x_n -  y_n \ge 2.
    \end{align*}
    for all $x \in \R^n_+\setminus \mathcal A$ and all $y \in \mathcal B$.
    Hence, due to \eqref{EST:B:GG:4}, we have 
    \begin{equation}
        \label{EST:GG:B1}
        \big| G_\gamma(x,y)(x-y) \big| \ge \frac{1}{2}|x-y| \ge 1.
    \end{equation}
    Hence, using Assumption \ref{ASS:J} and applying the change of variables $y \mapsto z=\frac{x-y}{\eps}$, we deduce that
    \begin{align}\label{EST:I11}
    \begin{aligned}
        \int_{\mathcal B}\rho_\eps\big( G_\gamma(x,y)(x-y) \big)  \dy 
        &\leq C\eps^{-n}\int_{\mathcal B}\Big|\frac{x-y}{\eps}\Big|^{2-\alpha-n}\Big(1+\Big|\frac{x-y}{\eps}\Big|\Big)^{-N}  \dy \\
        &= C\int_{\R^n}|z|^{2-\alpha-n}(1+|z|)^{-N}\dz \leq C.
    \end{aligned}
    \end{align}
    By \eqref{EST:B:DET}, \eqref{EST:I11} and Hölder's inequality, we find that
    \begin{align*}
        I_\eps^{1,1} &\leq C\int_{\R^n_+\setminus\mathcal A}|w(x)|^p\left(\int_{\mathcal B}\rho_\eps\big( G_\gamma(x,y)(x-y) \big)  \dy\right)^p
        \dx
        \leq C\|w\|_{L^p(\R^n_+)}^p.
    \end{align*}
    Proceding similarly, we use \eqref{EST:B:DET}, \eqref{EST:I11} and Hölder's inequality to obtain
    \begin{align*}
        I_\eps^{1,2} 
        &\leq C\int_{\R^n_+\setminus\mathcal A}
        \left(\int_{\mathcal B}\rho_\eps\big( G_\gamma(x,y)(x-y) \big)\dy\right)^{p-1}
        \\
        &\qquad\qquad\qquad
        \cdot \left(\int_{\mathcal B}\rho_\eps\big( G_\gamma(x,y)(x-y) \big)|w(y)|^p \dy\right)\dx
        \\
        &\leq C\int_{\R^n_+\setminus\mathcal A}
        \int_{\mathcal B}\;\rho_\eps\big( G_\gamma(x,y)(x-y) \big)|w(y)|^p \dy\dx.
    \end{align*}
    Finally, recalling Assumption \ref{ASS:J}, using Fubini's theorem, and preoceeding similarly to \eqref{EST:I11}, we conclude
    \allowdisplaybreaks
    \begin{align*}
        I_\eps^{1,2} 
        &\leq C\eps^{-n}\int_{\R^n_+\setminus\mathcal A}
        \int_{\mathcal B}\;\Big|\frac{x-y}{\eps}\Big|^{2-\alpha-n}\Big(1+\Big|\frac{x-y}{\eps}\Big|\Big)^{-N}|w(y)|^p \dy\dx \\
        &\leq C\eps^{-n}\int_{\mathcal B}|w(y)|^p\left(\int_{\R^n_+\setminus\mathcal A}\Big|\frac{x-y}{\eps}\Big|^{2-\alpha-n}\Big(1+\Big|\frac{x-y}{\eps}\Big|\Big)^{-N}\dx\right)\dy \\
        &\leq C\|w\|_{L^p(\R^n)}^p.
    \end{align*}
    \allowdisplaybreaks[0]
    In summary, this shows
    \begin{align*}
        I_\eps^{1} \leq C\|w\|_{L^p(\R^n)}^p.
    \end{align*}

    \medskip
    
    \noindent\textit{Ad $I_\eps^2$:}
    Here, we notice that for all $x\in \R^n_+$ and all $y \in \R^n_-\setminus \mathcal B$, it holds
    \begin{align*}
        \abs{ x - y} 
        \ge \abs{ x_n -  y_n} 
        =  x_n -  y_n \ge 2.
    \end{align*}
    Hence, $I_\eps^2$ can be estimated in a similar manner as $I_\eps^1$. In this way, we conclude that 
    \begin{align*}
        I_\eps^{2} \leq C\|w\|_{L^p(\R^n)}^p.
    \end{align*}

    \medskip
    
    \noindent\textit{Ad $I_\eps^3$:}
    A straightforward computation yields
    \begin{equation}
        \label{EST:I3}
        I_\eps^3 
        \leq C\big( I^{3,2}_\eps + I^{3,1}_\eps + I^{3,3}_\eps \big)
    \end{equation}   
    where
    \begin{align*}    
        &I^{3,1}_\eps \coloneqq
        \int_{\mathcal A}\Bigg|\int_{\mathcal B}
        J_\eps(G_\gamma(x,x)(x-y))(w(x)-w(y))\, d_{F_\gamma}(x)\dy\Bigg|^p\dx
        \\
        &I^{3,2}_\eps \coloneqq
        \int_{\mathcal A}\Bigg|\int_{\mathcal B}J_\eps\big( G_\gamma(x,y)(x-y) \big)
         (w(x)-w(y))\big(d_{F_\gamma}(y) - d_{F_\gamma}(x)\big)\dy\Bigg|^p\dx 
        \\
        &\begin{aligned}
        I^{3,3}_\eps \coloneqq
        \int_{\mathcal A}\Bigg|\int_{\mathcal B}\Big(J_\eps\big( G_\gamma(x,y)(x-y) \big)-J_\eps(G_\gamma(x,x)(x-y))\Big)
        \qquad\quad\\
        \cdot \, (w(x)-w(y))\, d_{F_\gamma}(x)\dy\Bigg|^p\dx
        \end{aligned}
    \end{align*}
    These terms will be handled separately. 
    
    \noindent\textit{Ad $I^{3,1}_\eps$:}
    Due to the construction of $F_\gamma$, there exist continuously differentiable functions $U: \R^n \to \text{SO}(n)$, $H: \R^n \to \mathrm{GL}_n(\R)$ and $H': \R^n \to \mathrm{GL}_{n-1}(\R)$
    with
    \begin{align}
    \label{DEF:B:H}
        H(x) = 
        \left(\begin{array}{ccc|c}
		& & & 0 \\
		& H'(x) & & \vdots \\
		& & & 0 \\ \hline
		0 & \ldots & 0 & 1
		\end{array}\right)
        \quad\text{for all $x\in\R^n$}
    \end{align}
    such that
    \begin{align}
        \label{eq:B:DF}
		G_\gamma(x,x) = \text{D}F_\gamma(x) = U(x)H(x) 
        \quad\text{for all $x\in\R^n$.}
	\end{align} 
    For more details on this decomposition, we refer to \cite[Proof of Corollary~2]{AbelsTerasawa2009}. In particular, we have
    \begin{align}
        \label{INV:B:H}
        H(x)^{-1} = 
        \left(\begin{array}{ccc|c}
		& & & 0 \\
		& H'(x)^{-1} & & \vdots \\
		& & & 0 \\ \hline
		0 & \ldots & 0 & 1
		\end{array}\right)
        \quad\text{for all $x\in\R^n$}
    \end{align}
    and since $\det U(x) = 1$ for all $x\in\R^n$, it further holds that
    \begin{equation*}
        \det H(x) = d_{F_\gamma(x)}
        \quad\text{and}\quad
        \det \big( H(x)^{-1} \big) = \big(d_{F_\gamma(x)}\big)^{-1}
        \quad\text{for all $x\in\R^n$.}
    \end{equation*}
    Recalling $G_\gamma(x,x) = DF_\gamma(x)$, we now apply Taylor's theorem to derive the estimate
    \begin{align}
        \label{Est:B:HOm:1}
        I^{3,1}_\eps
        \leq C\big( I_\eps^{3,3,1} + I_\eps^{3,3,2} \big), 
    \end{align}
    where
    \begin{align*}
        I_\eps^{3,1,1}&\coloneqq \int_{\mathcal A}\left|\int_{\mathcal B}J_\eps\big( DF_\gamma(x)(x-y) \big)\nabla w(x)\cdot(x-y)d_{F_\gamma}(x)\dy\right|^p\dx, 
        \\
        I_\eps^{3,1,2}&\coloneqq \int_{\mathcal A}\left|\int_{\mathcal B}J_\eps\big( DF_\gamma(x)(x-y) \big)R_2(x,y)d_{F_\gamma}(x)\dy\right|^p\dx, 
        \nonumber
    \end{align*}
    and the error term is given by 
    \begin{align}
        \label{DEF:B:R2}
        R_2(x,y) = \sum_{|\beta|=2}\frac{2}{\beta!}\left(\int_0^1(1-t)D^\beta w\big( y+t(x-y) \big)\dt\right)(x-y)^\beta.
    \end{align}

    \noindent\textit{Ad $I^{3,1,1}_\eps$:}
    Using the change of variables $y\mapsto z = x-y$, then
    $z\mapsto \tilde{z} = H(x)z$, and finally $\tilde{z}\mapsto y = x - \tilde{z}$, the term $I_\eps^{3,1,1}$ can be reformulated as
    \allowdisplaybreaks
    \begin{align*}
        I_\eps^{3,1,1}
        &=\int_{\mathcal A}\left|\int_{\R^{n-1}\times(x_n,x_n+2)}
        J_\eps\big( U(x) H(x) z \big)\nabla w(x)\cdot z \dz \; d_{F_\gamma}(x) \right|^p  \dx 
        \\
        &= \int_{\mathcal A}\left|\int_{\R^{n-1}\times(x_n,x_n+2) }
        J_\eps\big( U(x) \tilde{z} \big) 
        \, H(x)^{-T}\nabla w(x)\cdot \tilde{z} \,\mathrm d\tilde{z}\right|^p\dx 
        \\
        &= \int_{\mathcal A}\left|\int_{\mathcal B}
        J_\eps\big( U(x) (x-y) \big) 
        \, H(x)^{-T}\nabla w(x)\cdot (x-y) \dy\right|^p\dx 
        \\
        &= \int_{\mathcal A}\left|H(x)^{-T}\nabla w(x)\cdot
        \int_{\mathcal B} 
        J_\eps\big( U(x)(x-y) \big)\, (x-y)\dy\right|^p\dx.
    \end{align*}
    \allowdisplaybreaks[0]
    Thus, exploiting the structure of $H(x)^{-T}$, we use Condition~\eqref{COND:J:2} to infer that
    \begin{align*}
        I_\eps^{3,1,1}
        &= \int_{\mathcal A}
        \left|\partial_{x_n} w(x)
        \int_{\mathcal B}J_\eps\big( U(x)(x-y) \big)(x_n-y_n)\dy\right|^p\dx.
    \end{align*}
    Since $u\in C^2_c(\overline{\R^n_\gamma})\cap W^{2,p}_I({\R^n_\gamma})$, we deduce by means of the chain rule and \eqref{COND:B:G:GAMMA} that 
    \begin{align}\label{eq:B:c_bc}
        \partial_{x_n}w(x',0) 
        = \partial_{x_n}\big(u \circ F_\gamma\big)(x',0) 
        = - \nabla u(x^\prime,\gamma(x^\prime)) \cdot \mathbf{n}\big(x^\prime,\gamma(x^\prime)\big) 
        = 0
    \end{align}
    for all $x'\in\R^{n-1}$. Moreover, we have
    \begin{align*}
        |x_n| \le |x_n - y_n|
        \quad\text{for all $x\in \mathcal A$
        and $y\in \mathcal B$.}
    \end{align*}
    Thus, invoking the fundamental theorem of calculus, we have
    \begin{equation*}
        \partial_{x_n}w(x)
        = \partial_{x_n}w(x) - \partial_{x_n}w(x',0)
        = \int_0^1 \partial_{x_n}^2 w(x',tx_n) \, x_n \dt
    \end{equation*}
    for all $x = (x',x_n) \in\R^{n}$.
    Hence, we obtain
    \begin{align}
        \label{EST:B:I311}
        I_\eps^{3,1,1}
        &= \int_{\mathcal A}\left|\big(\partial_{x_n}w(x)-\partial_{x_n}w(x^\prime,0)\big)
        \int_{\mathcal B} \!\! J_\eps\big( U(x)(x-y) \big) \, (x_n-y_n)\dy\right|^p \!\!\mathrm dx 
        \nonumber\\
        &= \int_{\mathcal A}\left|\int_0^1
        \int_{\mathcal B}J_\eps\big( U(x)(x-y) \big)
        \, \partial_{x_n}^2w(x^\prime,tx_n) \, x_n \, (x_n-y_n)\dy\dt\right|^p\dx 
        \nonumber\\
        &\leq \int_{\mathcal A}\left(\int_0^1
        \int_{\mathcal B}\rho_\eps\big( U(x)(x-y) \big)\,
        \big| \partial_{x_n}^2w(x^\prime,tx_n) \big|\dy\dt\right)^p\dx 
        \nonumber\\
        &= \int_{\mathcal A}\left[ \left(\int_0^1\big| \partial_{x_n}^2w(x^\prime,tx_n) \big|\dt\right)
        \left(\int_{\mathcal B}\rho_\eps\big( U(x)(x-y) \big)\dy\right) \right]^p\dx 
    \end{align}
    Applying the change of variables $y\mapsto z = U(x)(x-y)$ along with Lemma~\ref{LEM:DIRAC}\ref{LEM:DIRAC:A}, we deduce that 
    \begin{align}
        \label{ID:RHO:U}
        &\int_{\mathcal B}\rho_\eps\big( U(x)(x-y) \big)\dy
        \le \int_{\R^n}\rho_\eps(z) \dz = \norm{\rho}_{L^1(\R^n)}.
    \end{align}
    Hence, recalling the definition of $\mathcal A$ and using the change of variables $t\mapsto s=tx_n$, we get 
    \begin{align}
        I_\eps^{3,1,1}
        \le C \int_{\R^{n-1}}  \int_0^\infty 
        \left(\frac{1}{x_n}\int_0^{x_n}|\partial_{x_n}^2w(x^\prime,s)|\ds\right)^p \, \mathrm dx_n \dx' 
    \end{align}
    Finally, applying Hardy's inequality, we conclude that
    \begin{align}
        \label{EST:I311}
        I_\eps^{3,1,1} 
        \leq C \left( \frac{p}{p-1} \right)^p \int_{\R^{n-1}}\int_0^\infty|\partial_{x_n}^2w(x)|^p\dx_n\dx^\prime  
        \leq C\|w\|_{W^{2,p}(\R^n_+)}^p.
    \end{align}

   
    \noindent\textit{Ad $I^{3,1,2}_\eps$:}
    Recalling the definition of $R_2$ (see~\eqref{DEF:B:R2}) and \eqref{EST:B:GG:4}, we use Hölder's inequality and Lemma~\ref{LEM:DIRAC}\ref{LEM:DIRAC:A} to obtain
    \begin{align}
        \label{EST:I312:0}
        I^{3,1,2}_\eps 
        &= \int_{\mathcal A}
        \left|\int_{\mathcal B}
        J_\eps\big( DF_\gamma(x)(x-y) \big) \, R_2(x,y) \, d_{F_\gamma}(x)\dy
        \right|^p\dx 
        \nonumber\\[1ex]
        &\leq C \sum_{|\beta|=2}  \int_{\mathcal A}
        \left(\int_0^1 \int_{\mathcal B}
        \rho_\eps\big( DF_\gamma(x)(x-y) \big) \,  |D^\beta w\big( y+t(x-y) \big)| \dy \dt 
        \right)^p\dx 
        \nonumber\\[1ex]
        &\leq C \sum_{|\beta|=2} \int_{\mathcal A}
        \left(
        \int_{\mathcal B}
        \rho_\eps\big( DF_\gamma(x)(x-y) \big) \dy
        \right)^{p-1}
        \\[-0.5ex]
        &\qquad\qquad\qquad \cdot \left(
        \int_0^1 \int_{\mathcal B}
        \rho_\eps\big( DF_\gamma(x)(x-y) \big)
        \, |D^\beta w\big( y+t(x-y) \big)|^p \dy \dt 
        \right) \dx. 
        \nonumber
    \end{align}
    Using the change of variables $y\mapsto DF_\gamma(x)(x-y)$ as well as Lemma~\ref{LEM:DIRAC}\ref{LEM:DIRAC:A}, we deduce that
    \begin{align*}
        \int_{\mathcal B}
        \rho_\eps\big( DF_\gamma(x)(x-y) \big) \dy
        \le 
        \int_{\R^n}
        \rho_\eps(z) \dz
        =
        \norm{\rho}_{L^1(\R^n)}.
    \end{align*}
    Consequently, we have
    \begin{align*}
        I^{3,1,2}_\eps 
        \leq C \sum_{|\beta|=2} \int_{\mathcal A}  \int_0^1 \int_{\mathcal B}
        \rho_\eps\big( DF_\gamma(x)(x-y) \big) \, |D^\beta w\big( y+t(x-y) \big)|^p \dy\dt\dx.
    \end{align*}    
    We now apply the change of variables 
    \begin{equation}
        \label{COV:XIETA}
        \R^{2n}\ni
        \begin{pmatrix}
            x \\ y
        \end{pmatrix}
        \mapsto
        \begin{pmatrix}
            \xi \\ \eta
        \end{pmatrix}
        \coloneqq
        \begin{pmatrix}
            x-y \\ y + t (x-y)
        \end{pmatrix}
        \in\R^{2n}.
    \end{equation}
    Note that 
    \begin{equation*}
        \det\Big( D_{(x,y)} 
        \left(
        \begin{smallmatrix}
            \xi \\ \eta
        \end{smallmatrix}
        \right)
        \Big)
        = \det
        \begin{pmatrix}
            I & -I \\
            tI & (1-t) I 
        \end{pmatrix}
        = 1.
    \end{equation*}
    In this way, we obtain
    \begin{align}
        \label{EST:I312:1}
        I^{3,1,2}_\eps 
        \leq C  \int_0^1 \sum_{|\beta|=2} \int_{\R^n} 
        \left( \int_{\R^n}
        \rho_\eps\Big( DF_\gamma\big(\eta + (1-t) \xi \big)\xi \Big) \,\mathrm d\xi 
        \right)
        \, |D^\beta w(\eta)|^p 
        \,\mathrm d\eta  \dt.
    \end{align}
    Applying \eqref{EST:B:GG:4}, we deduce that
    \begin{equation}
        \label{EST:I312:2}
        \Big| DF_\gamma\big(\eta + (1-t) \xi \big)\xi \Big| \ge \frac 12 |\xi|
    \end{equation}
    for all $\xi\in\R^n$ and $t\in(0,1)$. 
    Invoking Assumption~\ref{ASS:J} and applying the change of variables $\xi\mapsto \zeta =  \xi/\eps$, we infer that
    \allowdisplaybreaks
    \begin{align}
        \label{EST:I312:4}
        &\int_{\R^n} \rho_\eps\Big( DF_\gamma\big(\eta + (1-t) \xi \big)\xi \Big) \,\mathrm d\xi
        = 
        \int_{\R^n} \rho_\eps\Big( DF_\gamma\big(\eta + (1-t) \xi \big)\xi \Big) \,\mathrm d\xi
        \nonumber\\
        &\quad
        \le C\eps^{-n} \int_{\R^n}
        \Big|\frac{\xi}{\eps}\Big|^{2-\alpha-n} 
        \Big( 1 + \Big|\frac{\xi}{\eps}\Big| \Big)^{-N}\,\mathrm d\xi
        \nonumber\\
        &\quad
        = C\int_{\R^n}
        |\zeta|^{2-\alpha-n} 
        ( 1 + |\zeta| )^{-N}\,\mathrm d\zeta
        \le C.
    \end{align}
    \allowdisplaybreaks[0]
    Hence, in view of \eqref{EST:I312:1}, we obtain
    \begin{align*}
        \label{EST:I312}
        I^{3,1,2}_\eps 
        \leq C \sum_{|\beta|=2} \int_{\R^n} |D^\beta w(z)|^p \dz
        \leq C \|w\|_{W^{2,p}(\R^n)}^p.
    \end{align*}
    
    In summary, recalling \eqref{Est:B:HOm:1}, we finally conclude that
    \begin{equation}
        \label{EST:I31}
        I^{3,1}_\eps \leq C \|w\|_{W^{2,p}(\R^n)}^p.
    \end{equation}

    \noindent\textit{Ad $I^{3,2}_\eps$:}
    It follows from \eqref{EST:B:DFG} that $d_{F_\gamma}$ is Lipschitz continuous. Hence, using the fundamental theorem of calculus and Hölder's inequality, we deduce that
    \begin{align}
        \label{EST:I32:0}
        I^{3,2}_\eps 
        &=\int_{\mathcal A}\Bigg|\int_{\mathcal B} \!\! J_\eps\big( G_\gamma(x,y)(x-y) \big)
         (w(x)-w(y))\big(d_{F_\gamma}(y) - d_{F_\gamma}(x)\big)\dy\Bigg|^p \!\!\!\dx 
        \nonumber \\
        &\leq C\int_{\mathcal A}\left(\int_{\mathcal B}\int_0^1\rho_\eps\big( G_\gamma(x,y)(x-y) \big)\big|\nabla w\big( y+t(x-y) \big) \big|\dt \dy\right)^p\dx
        \nonumber \\
        &\leq C \int_{\mathcal A}
        \left(
        \int_{\mathcal B}
        \rho_\eps\big( G_\gamma(x,y)(x-y) \big) \dy
        \right)^{p-1}
        \\
        & \qquad\qquad
        \cdot \left(
        \int_0^1 \int_{\mathcal B}
        \rho_\eps\big( G_\gamma(x,y)(x-y) \big)
        \, \big|\nabla w\big( y+t(x-y) \big) \big|^p \dy \dt 
        \right) \dx 
        \nonumber
    \end{align}
    Applying the change of variables $y\mapsto z=x-y$, we obtain
    \begin{equation}
        \label{EST:I32:1}
        \int_{\mathcal B} \rho_\eps \big( G_\gamma(x,y)(x-y) \big) \dy
        \le \int_{\R^n} \rho_\eps \big( G_\gamma(x,x-z) z \big) \dz.
    \end{equation}
    Hence, proceeding similarly to \eqref{EST:I312:4}, we infer that
    \begin{align}
        \label{EST:I32:3}
        &\int_{\mathcal B} \rho_\eps \big( G_\gamma(x,y)(x-y) \big) \dy
        \le C.
    \end{align}
    This implies that 
    \begin{align*}
        I^{3,2}_\eps 
        \leq C \int_0^1 \int_{\R^n} \int_{\R^n}
        \rho_\eps\big( G_\gamma(x,y)(x-y) \big)
        \, \big|\nabla w\big( y+t(x-y) \big) \big|^p \dy \dx \dt 
    \end{align*}
    Applying the change of variables \eqref{COV:XIETA} and proceeding similarly to \eqref{EST:I312:1}--\eqref{EST:I312:4}, we finally conclude that
    \begin{equation}
        \label{EST:I32}
        I^{3,2}_\eps \leq C \|w\|_{W^{2,p}(\R^n)}^p.
    \end{equation}

    \noindent\textit{Ad $I^{3,3}_\eps$:} In order to estimate $I^{3,3}_\eps$, we first notice that the fundamental theorem of calculus yields
    \begin{align}
        \label{EST:DIFF:J}
        \begin{split}
        &J_\eps\big( G_\gamma(x,y)(x-y) \big)-J_\eps(G_\gamma(x,x)(x-y)) \\
        &\quad = \int_0^1 \nabla J_\eps(z_s) \cdot (G_\gamma(x,y)-G_\gamma(x,x))(x-y)\ds,
        \end{split}
    \end{align}
    for all $x,y\in\R^n$, where 
    \begin{align}
        \label{DEF:ZS}
        z_s = z_s(x,y)
        \coloneqq G_\gamma(x,x) (x-y)
            + s\big[ G_\gamma(x,y) - G_\gamma(x,x) \big] (x-y) 
    \end{align}
    for all $s\in [0,1]$. 
    Recalling the definition of $G_\gamma$ in \eqref{DEF:B:MG} and using the Lipschitz continuity of $DF_\gamma$, which follows from \eqref{EST:B:DFG}, we find that
    \begin{align}
        \label{EST:LIP:G}
        \big| G_\gamma(x,y)-G_\gamma(x,x) \big| \le C |x-y|
    \end{align}
    for all $x,y\in\R^n$.
    Plugging this estimate into \eqref{EST:DIFF:J}, we infer that
    \begin{align}
        \label{EST:DIFF:J*}
        \big| J_\eps\big( G_\gamma(x,y)(x-y) \big)-J_\eps(G_\gamma(x,x)(x-y)) \big| 
        \le C\int_0^1 |\nabla J_\eps(z_s)| \, |x-y|^2 \ds
    \end{align}
    for all $x,y\in\R^n$.
    Combining \eqref{EST:B:GG:2} and \eqref{EST:B:GG:3}, we further deduce that
    \begin{align}
        \label{EST:ZT}
        \begin{split}
        |z_s| 
        &\ge \Big| \big| G_\gamma(x,x) (x-y) \big|
            - s \big|\big[ G_\gamma(x,y) - G_\gamma(x,x) \big] (x-y) \big| \Big|
        \\[0.5ex]
        &\ge \frac 56 |x-y| - \frac 13 s \,|x-y|
        \ge \frac 12 |x-y|
        \end{split}
    \end{align}
    for all $x,y\in\R^n$.
    Using the chain rule, we derive the estimate
    \begin{align*}
        |\nabla J_\eps(x)| \le \frac{|\nabla \rho_\eps(x)|}{|x|^2} + 2 \frac{|\rho_\eps(x)|}{|x|^3}
    \end{align*}
    for all $x\in\R^n\setminus\{0\}$. In view of \eqref{EST:ZT}, we thus obtain
    \begin{align}
    \label{EST:JZS}
        |\nabla J_\eps(z_s)| \le C \frac{|\nabla \rho_\eps(z_s)|}{|x-y|^2} + C \frac{|\rho_\eps(z_s)|}{|x-y|^3}
    \end{align}
    for all $x,y\in\R^n\setminus\{0\}$.
    Invoking the fundamental theorem of calculus, we further have
    \begin{align}
        \label{EST:DIFF:W}
        w(x) - w(y) = \int_0^1 \nabla w\big( y+t(x-y) \big)\dt \cdot (x-y)
    \end{align}
    for all $x,y\in\R^n$. Combining \eqref{EST:DIFF:J*}, \eqref{EST:JZS} and \eqref{EST:DIFF:W}, we now conclude that
    \begin{align*}
        I^{3,3}_\eps 
        &\le C\int_{\mathcal A}\Bigg|\int_{\mathcal B} \int_0^1 \int_0^1
        |\nabla J_\eps(z_s)| \big| \nabla w\big( y+t(x-y) \big) \big| |x-y|^3 \ds \dt \dy \Bigg|^p \dx.
        \\[1ex]
        &
        \begin{aligned}
        \;\le C\int_{\mathcal A} \Bigg(\int_{\mathcal B} \int_0^1 \int_0^1
        &\big[ |\rho_\eps(z_s)| + |\nabla \rho_\eps(z_s)|\, |x-y| \big] 
        \\
        &\cdot \big| \nabla w\big( y+t(x-y) \big) \big| \ds \dt \dy \Bigg)^p \dx
        \end{aligned}
    \end{align*}
    Hence, by means of Hölder's inequality, we obtain
    \begin{align}
        \label{EST:I33:0}
        \begin{split}
        I^{3,3}_\eps 
        \le C\int_{\mathcal A} \Bigg(\int_{\mathcal B} \int_0^1 
        \big[ |\rho_\eps(z_s)| + |\nabla \rho_\eps(z_s)|\, |x-y| \big] \ds \dy \Bigg)^{p-1}&
        \\
        \cdot \Bigg(\int_{\mathcal B} \int_0^1 \int_0^1 
        \big[ |\rho_\eps(z_s)| + |\nabla \rho_\eps(z_s)|\, |x-y| \big] \qquad&
        \\
        \cdot \big| \nabla w\big( y+t(x-y) \big) \big|^p 
        \ds \dt \dy \Bigg) &\dx
        \end{split}
    \end{align}
    Due to \eqref{EST:ZT} and Condition~\eqref{COND:RHO:2} from Assumption~\ref{ASS:J}, we proceed as in the estimates for $I^{3,1,2}_\eps$ and $I^{3,2}_\eps$ to conclude that
    \begin{equation}
        \label{EST:I33}
        I^{3,3}_\eps \leq C \|w\|_{W^{2,p}(\R^n)}^p.
    \end{equation}

    Now, combining \eqref{EST:I31}, \eqref{EST:I32} and \eqref{EST:I33}, we infer from \eqref{EST:I3} that
    \begin{equation}
        \label{EST:I3*}
        I^{3}_\eps \leq C \|w\|_{W^{2,p}(\R^n)}^p.
    \end{equation}
    
    Since $I^1_\eps = I^2_\eps = 0$, we conclude from \eqref{EST:B:RE} and \eqref{EST:B:WU} that 
    \begin{align}
        \label{EST:B:RE*}
        \|\mathcal{R}_\eps\tilde{u}\|_{L^p({\R^n_\gamma})}^p
        \leq C \|w\|_{W^{2,p}(\R^n)}^p
        \leq C \|u\|_{W^{2,p}({\R^n_\gamma})}^p.
    \end{align}
    In view of \eqref{EST:B:1}, this means that \eqref{EST:LOM:CC2*} is finally verified.

    \medskip
    \noindent
    \textbf{Step~2: Proof of Estimate~\eqref{CONV:CHS} for $u\in C^3_c(\overline{{\R^n_\gamma}}) \cap W^{3,p}_I({\R^n_\gamma})$.}\\
    Now, let $u\in C^3_c(\overline{{\R^n_\gamma}}) \cap W^{3,p}_I({\R^n_\gamma})$.
    Our goal is to show that there exists a constant $C_\gamma>0$ depending only on ${\R^n_\gamma}$, $\rho$ and $p$ such that 
    \begin{equation*}
		\|\LL_\eps^{\R^n_\gamma} u + \Delta u\|_{L^p({\R^n_\gamma})} 
        \leq C_\gamma\, \sqrt[p]{\eps}\, \|u\|_{W^{3,p}({\R^n_\gamma})}.
    \end{equation*}
    Since ${\R^n_\gamma}$ is of class $C^3$, we can find an extension $\tilde{u} \in C^3_c(\R^n)$ with $u\vert_{\R^n_\gamma} = u$. In particular, $\tilde{u}$ can be chosen in such a way that
    \begin{equation}
        \label{EST:C:EXT}
        \norm{\tilde{u}}_{W^{3,p}(\R^n)} \le C \norm{u}_{W^{3,p}({\R^n_\gamma})}.
    \end{equation}    
    By means of Theorem~\ref{THM:CONV:R^N}, we derive the estimate
	\begin{align}
		\label{EST:C:1}
		\|\LL^{\R^n_\gamma}_\eps u + \Delta u\|_{L^p({\R^n_\gamma})}^p 
		&\leq C \|\LL^{\R^n}_\eps \tilde{u} + \Delta \tilde{u}\|_{L^p(\R^n)}^p 
		+ \|\mathcal{R}_\eps\tilde{u}\|_{L^p({\R^n_\gamma})}^p
		\nonumber\\
		&\leq C\eps^p\|\tilde{u}\|_{W^{3,p}(\R^n)}^p + \|\mathcal{R}_\eps\tilde{u}\|_{L^p({\R^n_\gamma})}^p
		\nonumber\\
		&\leq C\eps^p\|u\|_{W^{3,p}({\R^n_\gamma})}^p + \|\mathcal{R}_\eps\tilde{u}\|_{L^p({\R^n_\gamma})}^p,
	\end{align}
	where the error term is given by 
	\begin{align*}
		\mathcal{R}_\eps \tilde u(x) \coloneqq \int_{{\R^n_\gamma}^c}J_\eps(x-y)\big( u(x)-\tilde{u}(y) \big) \dy\quad \text{ for all $x\in{\R^n_\gamma}$}.
	\end{align*}
	Therefore, it remains to show that
	\begin{equation}
		 \label{EST:C:2}
		 \|\mathcal{R}_\eps\tilde{u}\|_{L^p({\R^n_\gamma})}^p 
		 \le C\eps\|u\|_{W^{3,p}({\R^n_\gamma})}^p
	\end{equation}
    since \eqref{EST:LOM:CC2*} then follows by combining \eqref{EST:B:1} and \eqref{EST:B:2}.
    To this end, we define the radius $R>0$, the functions $F_\gamma$, $G_\gamma$ and $w$ and the sets $\mathcal A$ and $\mathcal B$ as in Step~1.
    In particular, using the chain rule along with \eqref{EST:C:EXT}, \eqref{EST:B:DFG}and \eqref{EST:B:DET}, we deduce
    \begin{equation}
        \label{EST:C:WU} 
        \norm{w}_{W^{3,p}(\R^n)} 
        \le C \norm{\tilde{u}}_{W^{3,p}(\R^n)}
        \le C \norm{u}_{W^{3,p}({\R^n_\gamma})}.
    \end{equation}
    Moreover, recalling \eqref{EST:B:RE}, we have
    \begin{align}
        \label{EST:C:RE}
        \|\mathcal{R}_\eps\tilde{u}\|_{L^p({\R^n_\gamma})}^p \le C (I^1_\eps+I^2_\eps+I^3_\eps),
    \end{align}
    where $I^1_\eps,I^2_\eps,I^3_\eps$ are the integral terms introduced in Step~1.
    Now, as $N> 3-\alpha$ and higher regularity of $u$ is assumed, we intend to 
    improve the estimates from Step~1 for $I^{1}_\eps$, $I^{2}_\eps$ and $I^{3}_\eps$ such that the desired rate with respect to $\eps$ is obtained.

    \noindent\textit{Ad $I^1_\eps$:} 
    Recalling \eqref{EST:GG:B1}, it holds that  
    \begin{equation*}
        \big| G_\gamma(x,y)(x-y) \big| \ge \frac{1}{2}|x-y| \ge 1
    \end{equation*}
    for all $x \in \R^n_+\setminus \mathcal A$ and all $y \in \mathcal B$.
    Let $I_\eps^{1,1}$ and $I_\eps^{1,2}$ be defined as in Step~1.
    Using Assumption \ref{ASS:J}, we deduce that
    \begin{align}\label{EST:I11'}
        \int_{\mathcal B}\rho_\eps\big( G_\gamma(x,y)(x-y) \big)  \dy 
        &\leq C\eps\int_{\mathcal B}\rho_\eps\big( G_\gamma(x,y)(x-y) \big)\Big|\frac{x-y}{\eps}\Big|  \dy 
        \nonumber\\
        &\leq C\eps^{-n+1}\int_{\mathcal B}\Big|\frac{x-y}{\eps}\Big|^{3-\alpha-n}\Big(1+\Big|\frac{x-y}{\eps}\Big|\Big)^{-N}  \dy 
        \nonumber\\
        &= C\eps\int_{\R^n}|z|^{3-\alpha-n}(1+|z|)^{-N}\dz \leq C\eps,
    \end{align}
    since $N> 3-\alpha$.
    Consequently, we have
    \begin{align*}
        I_\eps^{1,1} &\leq C\int_{\R^n_+\setminus\mathcal A}|w(x)|^p\left(\int_{\mathcal B}\rho_\eps\big( G_\gamma(x,y)(x-y) \big)  \dy\right)^p
        \dx 
        \leq C\eps^p\|w\|_{L^p(\R^n_+)}^p.
    \end{align*}  
    To estimate $I^{1,2}_\eps$, we first use Hölder's inequality to obtain
    \begin{align*}
        I_\eps^{1,2} \leq C\int_{\R^n_+\setminus\mathcal A}
        &\left(\int_{\mathcal B}\rho_\eps\big( G_\gamma(x,y)(x-y) \big)\dy\right)^{p-1}
        \\
        &\cdot \left(\int_{\mathcal B}\rho_\eps\big( G_\gamma(x,y)(x-y) \big)|w(y)|^p \dy\right)\dx.
    \end{align*}
    Employing \eqref{EST:I11'}, we infer that 
    \begin{align*}
       I_\eps^{1,2} \leq C\eps^{1-\tfrac{1}{p}}\int_{\R^n_+\setminus\mathcal A}
        \left(\int_{\mathcal B}\rho_\eps\big( G_\gamma(x,y)(x-y) \big)|w(y)|^p \dy\right)\dx.
    \end{align*}
    Furthermore, recalling $N>3-\alpha$, Assumption \ref{ASS:J}, using Fubini's theorem, and proceeding similarly to \eqref{EST:I11'}, we deduce that
    \begin{align*}
        &\int_{\mathcal B}\rho_\eps\big( G_\gamma(x,y)(x-y) \big)|w(y)|^p \dy \\
        &\quad\leq C\eps^{-n+1}\int_{\R^n_+\setminus\mathcal A}
        \left(\int_{\mathcal B}\Big|\frac{x-y}{\eps}\Big|^{3-\alpha-n}\Big(1+\Big|\frac{x-y}{\eps}\Big|\Big)^{-N}|w(y)|^p \dy\right)\dx \\
        &\quad\leq C\eps^{-n+1}\int_{\mathcal B}|w(y)|^p\left(\int_{\R^n_+\setminus\mathcal A}\Big|\frac{x-y}{\eps}\Big|^{3-\alpha-n}\Big(1+\Big|\frac{x-y}{\eps}\Big|\Big)^{-N}\dx\right)\dy \\
        &\quad\leq C\eps\|w\|_{L^p(\R^n)}^p.
    \end{align*}
    Combining these estimates, we arrive at
    \begin{align*}
        I_\eps^{1,2} \leq C\eps\|w\|_{L^p(\R^n)}^p
    \end{align*}
    Altogether, this shows
    \begin{align}\label{EST:C:R1}
        I_\eps^{1} \leq C\eps\|w\|_{L^p(\R^n)}^p.
    \end{align}

    \noindent\textit{Ad $I^{2}_\eps$:} Arguing in a similar manner as for $I^{1}_\eps$, we obtain
    \begin{align}\label{EST:C:R2}
        I_\eps^{2} \leq C\eps\|w\|_{L^p(\R^n)}^p.
    \end{align}

    \noindent\textit{Ad $I^{3}_\eps$:} To derive a bound on $I^{3}_\eps$, we recall from Step~1 that
    \begin{align}
        I^{3}_\eps \le C\big( I^{3,1,1}_\eps + I^{3,1,2}_\eps + I^{3,2}_\eps + I^{3,3}_\eps \big).
    \end{align}
    Here, the integral terms on the right-hand side are defined as in Step~1 and will be handled separately.
    
    \noindent\textit{Ad $I^{3,1,1}_\eps$:} According to \eqref{EST:B:I311}, we have
    \begin{align*}
        I_\eps^{3,1,1}
        = C \int_{\mathcal A}\Bigg[ \left(\int_0^1\big| \partial_{x_n}^2w(x^\prime,tx_n) \big|\dt\right)
        \cdot \left(\int_{\mathcal B}\rho_\eps\big( U(x)(x-y) \big)\dy\right) \Bigg]^p\dx .
    \end{align*}   
    Recalling the definition of $\mathcal A$ and $\mathcal B$, we infer that
    \begin{align}
        \label{EST:C:I311:1}
        \begin{split}
        I_\eps^{3,1,1}
        \le C &\int_{\R^{n-1}} \norm{ \partial_{x_n}^2w(x^\prime,\cdot) }_{L^\infty((0,2))}^p
        \\
        &\quad\cdot \int_{0}^{2} \left(\int_{\R^{n-1}} \int_{-2}^0 \rho_\eps\big( U(x)(x-y) \big)
        \,\mathrm dy_n \dy' \right)^p  
        \,\mathrm dx_n \dx' .
        \end{split}
    \end{align}  
    
    Recalling the definition of $\rho_\eps$, we use the changes of variables ${y\mapsto z=x-y}$, ${z\mapsto \eps z}$, ${x_n\mapsto \eps x_n}$ and ${z\mapsto y = x-z}$ to compute
    \allowdisplaybreaks
    \begin{align}
        \label{EST:RATE:EPS}
        &\int_{0}^{2} \left(\int_{\R^{n-1}} \int_{-2}^{0} \rho_\eps\big( U(x)(x-y) \big)
        \,\mathrm dy_n \dy' \right)^p  
        \,\mathrm dx_n 
        \nonumber\\
        &\quad 
        = \int_{0}^{2} \left(\eps^{-n} \int_{\R^{n-1}} \int_{x_n}^{x_n + 2} \rho\big( U(x)\tfrac{z}{\eps} \big)
        \,\mathrm dz_n \dz' \right)^p  
        \,\mathrm dx_n 
        \nonumber\\
        &\quad 
        = \int_{0}^{2} \left( \int_{\R^{n-1}} \int_{x_n/\eps}^{(x_n + 2)/\eps} \rho\big( U(x)z \big)
        \,\mathrm dz_n \dz' \right)^p  
        \,\mathrm dx_n  
        \nonumber\\
        &\quad 
        = \eps \int_{0}^{2/\eps} \left( \int_{\R^{n-1}} \int_{x_n}^{x_n + (2/\eps)} \rho\big( U(x',\eps x_n)z \big)
        \,\mathrm dz_n \dz' \right)^p  
        \,\mathrm dx_n  
        \nonumber\\
        &\quad 
        \le \eps \int_{0}^{\infty} \left( \int_{\R^{n-1}} \int_{x_n}^{x_n + (2/\eps)} \rho\big( U(x',\eps x_n)z \big)
        \,\mathrm dz_n \dz' \right)^p  
        \,\mathrm dx_n  
        \nonumber\\
        &\quad 
        = \eps \int_{0}^{\infty} \left( \int_{\R^{n-1}} \int_{-2/\eps}^{0} \rho\big( U(x',\eps x_n)(x-y)) \big)
        \,\mathrm dy_n \dy' \right)^p  
        \,\mathrm dx_n  \,.
    \end{align}
    \allowdisplaybreaks[0]
    for all $x'\in \R^{n-1}$.
    Since $U(z) \in \mathrm{SO}(n)$ for all $z\in\R^n$, we know that 
    \begin{equation*}
        \big| U(x',\eps x_n)(x-y) \big| = |x-y|
        \quad\text{for all $x,y\in\R^n$.}
    \end{equation*}
    Furthermore, we notice that for any $x, y \in\R^n$ with $x_n\ge 0$ and $y_n\le 0$, it holds that
    \begin{align*}
        |x-y| \ge |x_n - y_n| = x_n + |y_n| \ge x_n.
    \end{align*}
    Moreover, since $N> 3-\alpha$, we can find a $\delta>0$ such that $N> 3-\alpha+\delta$.
    Based on \eqref{EST:RATE:EPS}, we deduce that 
    \begin{align}
        \label{EST:AEPS}
        &\int_{0}^{2} \left(\int_{\R^{n-1}} \int_{-2}^{0} \rho_\eps\big( U(x)(x-y) \big)
        \,\mathrm dy_n \dy' \right)^p  
        \,\mathrm dx_n
        \nonumber\\
        &\quad
        \le  \eps\int_{0}^{\infty} \left( \int_{\R^{n-1}} \int_{-\infty}^{0} \rho\big( U(x',\eps x_n)(x-y)) \big)
        \,\mathrm dy_n \dy' \right)^p  
        \,\mathrm dx_n  
        \nonumber\\
        &\quad
        \le C \eps\int_{0}^{\infty} \left( \int_{\R^{n-1}} \int_{-\infty}^{0} 
        |x-y|^{2-\alpha-n} (1+|x-y|)^{-N} 
        \,\mathrm dy_n \dy' \right)^p  
        \,\mathrm dx_n 
        \nonumber\\
        &\quad
        \le C \eps\int_{0}^{\infty} \Bigg( \int_{\R^{n-1}} \int_{-\infty}^{0} 
        |x-y|^{2-\alpha-n} (1+|x-y|)^{-N+1+\delta} 
        \nonumber\\[-1ex]
        &\qquad\qquad\qquad\qquad\qquad\qquad\qquad\qquad
        \cdot (1+x_n)^{-(1+\delta)} 
        \,\mathrm dy_n \dy' \Bigg)^p  
        \,\mathrm dx_n
        \nonumber\\
        &\quad
        \le 
        C \eps \left( \int_{\R^{n}}
        |y|^{2-\alpha-n} (1+|y|)^{-N+1+\delta}
        \,\mathrm dy \right)^p  
        \int_{0}^{\infty} (1+x_n)^{-(1+\delta)p}
        \,\mathrm dx_n
        \;\le C\eps
    \end{align} 
    for all $x'\in \R^{n-1}$.
    Here, the integral with respect to $y$ exists since $N>3-\alpha+\delta$, and the integral with respect to $x_n$ exists since $-(1+\delta)p<-1$.
    Plugging this estimate into \eqref{EST:C:I311:1}, we infer that 
    \begin{align*}
        I_\eps^{3,1,1}
        \le C\eps \int_{\R^{n-1}} \norm{ \partial_{x_n}^2w(x^\prime,\cdot) }_{L^\infty((0,2))}^p \dx' .
    \end{align*}
   As the embedding $W^{1,p}((0,2)) \hookrightarrow L^\infty((0,2))$ is continuous, this implies that
    \begin{align}
        \label{EST:C:I311}
        \begin{split}
        I_\eps^{3,1,1}
        &\le C\eps \int_{\R^{n-1}} \norm{ \partial_{x_n}^2w(x^\prime,\cdot) }_{W^{1,p}((0,2))}^p \dx' 
        \le C\eps \, \norm{w}_{W^{3,p}(\R^n)}^p \,.
        \end{split}
    \end{align}

    \noindent\textit{Ad $I^{3,1,2}_\eps$:}
    We already know from \eqref{EST:B:GG:4} that
    \begin{equation}
        \label{EST:DFG}
        \big| DF_\gamma(x)z \big| \ge \frac 12 |z|
        \quad\text{for all $x,z\in\R^n$.}
    \end{equation}
    Moreover, for any $x,y\in \R^n$ with $y_n<0<x_n$, we have $|x_n-y_n| =  |y_n| + |x_n|$. Hence a straightforward computation yields
    \begin{align}
        \label{EQU:NORM}
        \big|(x'-y',y_n,x_n)\big| \le |x-y|
        \quad\text{and}\quad
        |x-y| \le \sqrt{2} \,\big|(x'-y',y_n,x_n)\big|
    \end{align}
    for all $x,y\in \R^n$ with $y_n<0<x_n$.
    Thus, by definition of $\rho_\eps$ in \ref{ASS:RHO}, \eqref{EST:DFG} and Assumption \ref{ASS:J}, we infer that 
    \begin{align*}
        |\rho_\eps(DF_\gamma(x))(x-y)| 
        &\leq \frac{C}{\eps^n}\Big|\frac{x-y}{\eps}\Big|^{2-\alpha-n}\Big(1+\Big|\frac{x-y}{\eps}\Big|\Big)^{-N} 
        \\
        &\leq \frac{C}{\eps^n}\Big|\Big(\frac{x^\prime-y^\prime}{\eps},\frac{y_n}{\eps}, \frac{x_n}{\eps}\Big)\Big|^{2-\alpha-n}\Big(1+\Big|\Big(\frac{x^\prime-y^\prime}{\eps},\frac{y_n}{\eps}, \frac{x_n}{\eps}\Big)\Big|\Big)^{-N}
    \end{align*}
    for all $x,y\in \R^n$ with $y_n<0<x_n$. To simplify the notation, we now introduce the function
    \begin{align*}
        g:\R^{n+1} \to \R,
        \quad g(\zeta) \coloneqq |\zeta|^{2-\alpha-n}(1+|\zeta|)^{-N} .
    \end{align*}
    Recalling the definitions of $\mathcal A$ and $\mathcal B$, we then deduce that 
    \begin{align}
    \label{EST:I312*}
        I_\eps^{3,1,2} \leq 
        &C\sum_{|\beta|=2}\int_{\R^{n-1}}\int_0^2\Bigg(\eps^{-n}\int_0^1\int_{\R^{n-1}}\int_{-2}^0 g\Big(\frac{x^\prime-y^\prime}{\eps},\frac{y_n}{\eps}, \frac{x_n}{\eps}\Big)
        \nonumber\\
        & \qquad \qquad \qquad \qquad  
            \cdot |D^\beta w(y+t(x-y))| \dy_n \dy' \dt\Bigg)^p \dx_n \dx'
        \nonumber\\
        &\leq C\sum_{|\beta|=2}\int_{\R^{n-1}}\int_0^2\Bigg(\eps^{-n}\int_0^1\int_{\R^{n-1}}\int_{-2}^0 g\Big(\frac{x^\prime-y^\prime}{\eps},\frac{y_n}{\eps}, \frac{x_n}{\eps}\Big)
        \nonumber\\
        & \qquad 
            \cdot \|D^\beta w(y^\prime+t(x^\prime-y^\prime),.)\|_{L^\infty((-2,2))}
            \dy_n \dy' \dt\Bigg)^p
        \dx_n \dx' .
    \end{align}
    Using the change of variables $y^\prime \mapsto z^\prime = x^\prime-y^\prime$ as well as $z^\prime\mapsto \eps z^\prime$, $x_n\mapsto \eps x_n$, and $y_n\mapsto \eps z_n$, it follows that 
    \begin{align*}
        I_\eps^{3,1,2} \leq 
        &C\eps\sum_{|\beta|=2} \int_{0}^{\infty} \int_{\R^{n-1}}
            \Bigg( \int_0^1 \int_{-\infty}^{0} \int_{\R^{n-1}}
            g(z',z_n,x_n)
        \\
        &\qquad \qquad \qquad \cdot 
            \|D^\beta w(x^\prime - \eps z^\prime +t\eps z^\prime,.)\|_{L^\infty((-2,2))}\dz' \dz_n \dt\Bigg)^p \dx' \dx_n .
    \end{align*}
    Now, the change of variables $x^\prime \mapsto x^\prime - \eps z^\prime +t\eps z^\prime$ yields
    \begin{align}
    \label{EST:I312**}
        I_\eps^{3,1,2} 
        &\leq 
        C\eps\sum_{|\beta|=2}
            \int_{0}^{\infty} \!\!\!\int_{\R^{n-1}} \!\Bigg(\int_{\R^n}g(z',z_n,x_n)\|D^\beta w(x^\prime,.)\|_{L^\infty((-2,2))}\dz' \dz_n \Bigg)^p \!\!\!\dx' \!\dx_n 
        \nonumber\\
        &\leq C\eps\|w\|_{W^{3,p}(\R^n)}^p \int_{0}^{\infty}
            \Bigg(\int_{-\infty}^{0} \int_{\R^{n-1}} g(z',z_n,x_n)\dz'\dz_n \Bigg)^p\dx_n.
    \end{align}
    Next, we introduce the auxiliary function
    \begin{align*}
        G:\R\to\R, \quad
        G(s) \coloneqq \int_{-\infty}^{0} \int_{\R^{n-1}} g(z',z_n,s) \dz' \dz_n.
    \end{align*}
    Using interpolation, we obtain
    \begin{align*}
        &\int_{0}^{\infty} \Bigg(\int_{-\infty}^{0} \int_{\R^{n-1}} g(z',z_n,x_n) \dz' \dz_n \Bigg)^p  \dx_n 
        \\[1ex]
        &\quad= \|G\|_{L^p((0,\infty)})^p \leq \|G\|_{L^1((0,\infty))}\|G\|_{L^\infty((0,\infty))}^{p-1}.
    \end{align*}
    Recalling the definitions of $G$ and $g$, we deduce that
    \begin{align*}
        \|G\|_{L^1((0,\infty))}
        \le \int_{\R^{n+1}} |\zeta|^{2-\alpha-n} (1+|\zeta|)^{-N} \,\mathrm{d}\zeta
        \le C
    \end{align*}
    since $N>3-\alpha$. Moreover, recalling \eqref{EQU:NORM} and applying the changes of variables $z'\mapsto x'-z'$ and $z\mapsto x-z$, we get
    \begin{align*}
        &\|G\|_{L^\infty((0,\infty))}
        \\
        &\;\;= \underset{x_n\in (0,\infty)}{\sup} \int_{-\infty}^{0} \int_{\R^{n-1}} 
            |(x'-z',z_n,x_n)|^{2-\alpha-n} 
            (1+|(x'-z',z_n,x_n)|)^{-N} 
            \dz' \dz_n
        \\
        &\;\;\le \underset{x_n\in (0,\infty)}{\sup} C \int_{\R^{n}}  
            |x-z|^{2-\alpha-n} (1+|x-z|)^{-N} 
            \dz
        \\
        &\;\;= C \int_{\R^{n}}  
            |z|^{2-\alpha-n} (1+|z|)^{-N} 
            \dz
        \le C.
    \end{align*}
    Finally, we conclude that 
    \begin{equation}
        \label{EST:C:I312}
        I^{3,1,2}_\eps \le C\eps \, \norm{w}_{W^{3,p}(\R^n)}^{p}.
    \end{equation}
    
    \medskip

    \noindent\textit{Ad $I^{3,2}_\eps$:} 
    We have already shown in \eqref{EST:I32:0} that
    \begin{align*}
        I^{3,2}_\eps
         \leq C &\int_{\mathcal A}
            \left(
            \int_{\mathcal B}
            \rho_\eps\big( G_\gamma(x,y)(x-y) \big) \dy
            \right)^{p-1}
            \\
            & \cdot \left(
            \int_0^1 \int_{\mathcal B}
            \rho_\eps\big( G_\gamma(x,y)(x-y) \big)
            \, \big|\nabla w\big( y+t(x-y) \big) \big|^p \dy \dt 
            \right) \dx .
    \end{align*}
    As we know from \eqref{EST:B:GG:4} that
    \begin{equation*}
        \big| DF_\gamma(x)z \big| \ge \frac 12 |z|
        \quad\text{for all $x,z\in\R^n$,}
    \end{equation*}
    the integral term $I^{3,2}_\eps$ can be estimated analogously to $I^{3,1,2}_\eps$. In this way, we obtain
    \begin{equation}
        \label{EST:C:I32}
        I^{3,2}_\eps \le C\eps \, \norm{w}_{W^{3,p}(\R^n)}^{p}.
    \end{equation}    

    \noindent\textit{Ad $I^{3,3}_\eps$:}
    According to \eqref{EST:I33:0}, we have
    \begin{align*}
        I^{3,3}_\eps 
        &\le C\int_{\mathcal A} \Bigg(\int_{\mathcal B} \int_0^1 
        \big[ |\rho_\eps(z_s)| + |\nabla \rho_\eps(z_s)|\, |x-y| \big] \ds \dy \Bigg)^{p-1}
        \\
        &\qquad\qquad
        \cdot \Bigg(\int_{\mathcal B} \int_0^1 \int_0^1 
        \big[ |\rho_\eps(z_s)| + |\nabla \rho_\eps(z_s)|\, |x-y| \big] 
        \\
        &\qquad\qquad\qquad\qquad\qquad
        \cdot \big| \nabla w\big( y+t(x-y) \big) \big|^p 
        \ds \dt \dy \Bigg) \dx ,
    \end{align*}
    where $z_s$ is given by \eqref{DEF:ZS}, that is,
    \begin{align*}
        z_s = z_s(x,y)
        \coloneqq \tilde{G}_s(x,y) (x-y)         
    \end{align*}
    with
    \begin{align*}
    \tilde G_s(x,y)
        \coloneqq G_\gamma(x,x) 
            + s \big[G_\gamma(x,y) - G_\gamma(x,x) \big]
    \end{align*}
    Moreover, we have shown in \eqref{EST:ZT} that 
    \begin{align}
        \label{EST:TGSXY}
        \big|\tilde G_s(x,y)(x-y)\big| = |z_s|  \ge \frac 12 |x-y| 
    \end{align}
    for all $x,y\in\R^n$. 
    Hence, recalling the definitions of $\mathcal A$ and $\mathcal B$ as well as \ref{ASS:J}, we use the change of variables $y\mapsto z=x-y$ to deduce that
    \begin{align}
        \label{EST:I33*}
        I^{3,3}_\eps 
        &\le C\int_{\R^{n-1}} \int_0^2 \Bigg( \int_{\R^{n-1}} \int_{-2}^0 \int_0^1 
        \big[ |\rho_\eps(z_s)| + |\nabla \rho_\eps(z_s)|\, |x-y| \big] \ds \dy_n \dy'  \Bigg)^{p-1} \!\!\!\dx_n
        \nonumber\\
        &\qquad\qquad
        \cdot \Bigg(\int_{\R^{n-1}} \int_0^4 
        \abs{\frac{z}{\eps}}^{2-\alpha-n} 
        \Big( 1 + \abs{\frac{z}{\eps}} \Big)^{-N}
        \\
        &\qquad\qquad\qquad
        \cdot \int_0^1 \big\| \nabla w\big( x'-z'+tz', \cdot \big) \big\|_{L^\infty((-2,2))}^p 
        \dt \dz_n \dz' \Bigg) \dx' 
        \nonumber
    \end{align}
    Now, we fix an arbitrary $x'\in \R^{n-1}$ and proceed similarly to \eqref{EST:RATE:EPS} and \eqref{EST:AEPS}.
    Recalling $\eps\in (0,1]$ and using the growth assumptions from \ref{ASS:J} as well as the changes of variables $y \mapsto x-z$ and $z\mapsto \eps z$, we obtain
    \allowdisplaybreaks
    \begin{align*}
        &\int_0^2 \Bigg( \int_{\R^{n-1}} \int_{-2}^0 \int_0^1 
        \big[ |\rho_\eps(z_s)| + |\nabla \rho_\eps(z_s)|\, |x-y| \big] \ds \dy_n \dy' \Bigg)^{p-1} \dx_n
        \\
        &\begin{aligned}
        &\le C \int_0^2 \Bigg( \eps^{-n} \int_{\R^{n-1}} \int_{x_n}^{x_n+2} \int_0^1 
        \abs{\tilde G_s(x,x-z)\tfrac{z}{\eps}}^{2-\alpha-n} 
        \Big( 1 + \abs{\tilde G_s(x,x-z)\tfrac{z}{\eps}} \Big)^{-N} 
        \\
        &\qquad\quad 
        + \abs{\tilde G_s(x,x-z)\tfrac{z}{\eps}}^{1-\alpha-n}
        \Big( 1 + \abs{\tilde G_s(x,x-z)\tfrac{z}{\eps}} \Big)^{-N} 
        \abs{\tfrac{z}{\eps}}
        \ds \dz_n\dz' \Bigg)^{p-1} \!\!\!\dx_n
        \end{aligned}
        \\
        &\begin{aligned}
        &\le C \int_0^2 \Bigg( \int_{\R^{n-1}} \int_{x_n/\eps}^{(x_n+2)/\eps} \int_0^1 
        \abs{\tilde G_s(x,x-\eps z)z}^{2-\alpha-n} 
        \Big( 1 + \abs{\tilde G_s(x,x-\eps z)z} \Big)^{-N} 
        \\
        &\qquad\quad 
        + \abs{\tilde G_s(x,x-\eps z)z}^{1-\alpha-n}
        \Big( 1 + \abs{\tilde G_s(x,x-\eps z)z} \Big)^{-N} 
        \abs{z}
        \ds \dz_n\dz' \Bigg)^{p-1} \!\!\!\!\!\dx_n.
        \end{aligned}
    \end{align*}
    \allowdisplaybreaks[0]%
    Next, employing the changes of variables $x_n \mapsto \eps x_n$ and $z\mapsto y = x-z$, using the abbreviation $x_\eps = (x',\eps x_n)$, and invoking \eqref{EST:TGSXY}, we infer that
    \allowdisplaybreaks
    \begin{align*}
        &\int_0^2 \Bigg( \int_{\R^{n-1}} \int_{-2}^0 \int_0^1 
        \big[ |\rho_\eps(z_s)| + |\nabla \rho_\eps(z_s)|\, |x-y| \big] \ds \dy_n\dy' \Bigg)^{p-1} \dx_n
        \\
        &\begin{aligned}
        &\le C\eps \! \int_0^2 \! \Bigg( \int_{\R^{n-1}} \int_{x_n}^{x_n+(2/\eps)} \! \int_0^1 
        \! \abs{\tilde G_s(x_\eps,x_\eps-\eps z)z}^{2-\alpha-n} 
        \Big( 1 + \abs{\tilde G_s(x_\eps,x_\eps-\eps z)z} \Big)^{-N} 
        \\
        &\quad 
        + \abs{\tilde G_s(x_\eps,x_\eps-\eps z)z}^{1-\alpha-n}
        \Big( 1 + \abs{\tilde G_s(x_\eps,x_\eps-\eps z)z} \Big)^{-N} 
        \abs{z}
        \ds \dz_n\dz' \Bigg)^{p-1} \dx_n.
        \end{aligned}
        \\
        &\begin{aligned}
        &\le C\eps \int_0^2 \Bigg( \int_{\R^{n-1}} \int_{x_n}^{x_n+(2/\eps)} 
        \abs{z}^{2-\alpha-n} 
        ( 1 + \abs{z} )^{-N} 
        \dz_n \dz' \Bigg)^{p-1} \dx_n.
        \end{aligned}
        \\
        &\begin{aligned}
        &\le C\eps \int_0^\infty \Bigg( \int_{\R^{n-1}} \int_{-\infty}^{0} 
        \abs{x-y}^{2-\alpha-n} 
        ( 1 + \abs{x-y} )^{-N} 
        \dy_n\dy' \Bigg)^{p-1} \dx_n.
        \end{aligned}
    \end{align*}
    \allowdisplaybreaks[0]%
    Arguing as in \eqref{EST:AEPS}, we deduce that the integral term in the last line is finite. Hence, we conclude that
    \begin{align*}
        &\int_0^2 \Bigg( \int_{\R^{n-1}} \int_{-2}^0 \int_0^1 
        \big[ |\rho_\eps(z_s)| + |\nabla \rho_\eps(z_s)|\, |x-y| \big] \ds \dy_n\dy' \Bigg)^{p-1} \dx_n
        \le C\eps
    \end{align*}
    for all $x'\in \R^{n-1}$. Using this inequality to estimate the right-hand side of \eqref{EST:I33*}, we infer\vspace{-1ex}
    \begin{align*}
        I^{3,3}_\eps 
        &\le C\eps \int_{\R^{n-1}} \int_{\R^{n-1}} \int_0^4 
        \abs{\frac{z}{\eps}}^{2-\alpha-n} 
        \Big( 1 + \abs{\frac{z}{\eps}} \Big)^{-N}
        \\
        &\qquad\qquad
        \cdot \int_0^1 \big\| \nabla w\big( x'-z'+tz', \cdot \big) \big\|_{L^\infty((-2,2))}^p 
        \dt \dz_n \dz' \dx' 
        \nonumber
    \end{align*}
    This integral term can be handled analogously to \eqref{EST:I312*} and \eqref{EST:I312**}. In this way, we finally conclude the estimate
    \begin{equation}
        \label{EST:C:I33}
        I^{3,3}_\eps \le C\eps \, \norm{w}_{W^{3,p}(\R^n)}^{p}.
    \end{equation}  
    
    In summary, collecting \eqref{EST:C:I311}, \eqref{EST:C:I312}, \eqref{EST:C:I32} and \eqref{EST:C:I33} we arrive at
    \begin{align}
        \label{EST:C:R3}
        I^{3}_\eps \le C\eps \, \norm{w}_{W^{3,p}(\R^n)}^{p}.
    \end{align}
    
    Finally, plugging the estimates \eqref{EST:C:R1}, \eqref{EST:C:R2}, and \eqref{EST:C:R3} into \eqref{EST:C:RE}, we have
    \begin{equation*}
        \|\mathcal{R}_\eps\tilde{u}\|_{L^p({\R^n_\gamma})}^p 
        \le C\eps \, \norm{w}_{W^{3,p}(\R^n)}^p,
    \end{equation*}
    which verifies \eqref{EST:C:2}.
    Combining this estimate with \eqref{EST:C:1} and \eqref{EST:C:WU}, we finally conclude that
    \begin{equation*}
        \|\mathcal{R}_\eps\tilde{u}\|_{L^p({\R^n_\gamma})}^p 
        \le C\eps \, \norm{u}_{W^{3,p}({\R^n_\gamma})}^p.
    \end{equation*}
    This means that \eqref{CONV:CHS} is established for $u\in C^3_c(\overline{{\R^n_\gamma}}) \cap W^{3,p}_I({\R^n_\gamma})$.

    \smallskip
    
    \noindent
    \textbf{Step~3: Construction of the operator and completion of the proof.}\\
    As the inclusion $C^2_c(\overline{{\R^n_\gamma}}) \cap W^{2,p}_I({\R^n_\gamma}) \subset W^{2,p}_I({\R^n_\gamma})$ is dense, we can use the result established in Step~1 to extend
    \begin{equation*}
        \LL_\eps^{\R^n_\gamma}: C^2_c(\overline{\R^n_\gamma}) \cap W^{2,p}_I({\R^n_\gamma}) \to L^p_\mathrm{loc}({\R^n_\gamma})
    \end{equation*}
    to a bounded linear operator\vspace{-1ex}
    \begin{equation}
        \label{DEF:LOM:HS*}
        \LL_\eps^{\R^n_\gamma}: W^{2,p}_I({\R^n_\gamma}) \to L^p({\R^n_\gamma})
    \end{equation}
    with
    \begin{equation}
        \label{BND:LOM:HS*}
        \norm{\LL_\eps^{\R^n_\gamma} u}_{L^p({\R^n_\gamma})} \leq C_\gamma \|u\|_{W^{2,p}({\R^n_\gamma})}
    \end{equation}
    for a suitable constant $C_\gamma>0$ that depends only on ${\R^n_\gamma}$, $\rho$ and $p$. In particular, the extension \eqref{DEF:LOM:HS*} is unique.
    We further know from Step~2 that \eqref{CONV:CHS}
    holds for all $u\in C^3_c(\overline{{\R^n_\gamma}}) \cap W^{3,p}_I({\R^n_\gamma})$. 
    Because of density, it is clear that \eqref{CONV:CHS} holds true for all $u\in W^{3,p}_I({\R^n_\gamma})$. Finally, proceeding similarly to the proof of Theorem~\ref{THM:CONV:R^N}\ref{THM:CONV:R^N:A}, we use the Banach--Steinhaus theorem to conclude the convergence \eqref{CONV:CHS:2}. 

    This means that all claims are established and thus, the proof is complete.
\end{proof}

\medskip


\begin{proof}[Proof of Corollary~\ref{COR:CONV:HS}]
    First, let $u\in C^2_c(\overline{Q\R^n_\gamma}) \cap W^{2,p}_Q(Q\R^n_\gamma)$ be arbitrary. We now introduce the function 
    \begin{equation*}
        w: \R^n_\gamma \to \R, \quad
        w(x) = u(Qx).
    \end{equation*}
    As the transformation $x\mapsto Qx$ is linear, it is clear that $w\in W^{3,p}(\R^n_\gamma)$. In particular, since $Q \in \mathrm{SO}(n)$, we obtain
    \begin{equation}
        \label{ID:W2P}
        \|w\|_{W^{2,p}(\R^n_\gamma)} = \|u\|_{W^{2,p}(Q\R^n_\gamma)}.
    \end{equation}    
    Moreover, using the chain rule, we obtain that
    \begin{align*}
        \nabla w(x) \cdot e_n 
        = Q\, \nabla u(Qx) \cdot e_n = 0 
        \quad\text{for all $x\in \partial \R^n_\gamma$}
    \end{align*}
    since $Qx \in \partial Q\R^n_\gamma$. Thus, we have $w\in W^{3,p}_I(\R^n_\gamma)$. Moreover, as the Laplace operator is invariant under rotations, we have
    \begin{equation*}
        \Delta w(x) = \Delta u(Qx)
        \quad\text{for all $x\in \R^n_\gamma$.}
    \end{equation*}
    This can also be easily verified by means of the chain rule. We now define $\widehat{\rho}:\R^n \to\R$ with
	\begin{equation*}
		\widehat{\rho}(x) = \rho(Qx) 
		\quad\text{for all $x\in\R^n\setminus\{0\}$}. 
	\end{equation*}
	Hence, its associated kernel is given by $\widehat{J} = J(Q\,\cdot\,)$. Accordingly, we write $\widehat{J}_\eps = J_\eps(Q\,\cdot\,)$ for all $\eps>0$.
	Since $\rho$ and $J$ satisfy \ref{ASS:RHO} and \ref{ASS:J} and $QQ^T = I$, it is straightforward to check that $\widehat{\rho}$ and $\widehat{J}$ fulfill \ref{ASS:RHO} and \ref{ASS:J} with
	\begin{align*}
        \begin{split}
		\widehat M 
        &\coloneqq \frac{1}{2} \int_{\R^n} \widehat{J}(z)\, z \otimes z \dz 
        = \frac{1}{2}\, Q \int_{\R^n} J(Qz)\, (Qz) \otimes (Qz) \dz \, Q^T
        \\
        &= \frac{1}{2}\, Q \int_{\R^n} J(Qz)\, (Qz) \otimes (Qz) \dz \, Q^T 
        = \frac{1}{2}\, Q \int_{\R^n} J(z)\, z \otimes z \dz \, Q^T 
        = I,
        \end{split}
		\\[1ex]
		\widehat A &\coloneqq  I
	\end{align*}
    in place of $M$ and $A$. In particular, according to Theorem~\ref{THM:CONV:HS} and Corollary~\ref{COR:WD:DOM}, the nonlocal operators
    \begin{align*}
        \begin{split}
		&\widehat{\LL}^{\R^n_\gamma}_\eps : W^{2,p}(\R^n_\gamma) \to L^p(\R^n_\gamma),
		\\
		&\widehat{\LL}^{\R^n_\gamma}_\eps v(x) = \mathrm{P.V.}\int_{\R^n_\gamma} \widehat{J}_\eps(x-y) \big( v(x) - v(y) \big) \dy
        \end{split}
	\end{align*}
    and
    \begin{align*}
        \begin{split}
            &\LL_\eps^{Q\R^n_\gamma}: 
            C_c^2(\overline{Q\R^n_\gamma}) \cap W^{2,p}_Q(Q\R^n_\gamma) 
            \to L^p_{\mathrm{loc}}(Q\R^n_\gamma),
            \\
            &\LL_\eps^{Q\R^n_\gamma} v(x) 
            = \mathrm{P.V.} \int_{Q\R^n_\gamma} J_\eps(x-y) \big( v(x) - v(y) \big) \dy
        \end{split}
    \end{align*}
    are well-defined for every $\eps>0$. 
    Employing the changes of variables $x\mapsto Qx$ and $y\mapsto Qy$, we deduce that 
    \begin{align}
		\label{TRAFO:Q}
		\begin{split}
			&\big\|\mathcal{L}^{Q\R^n_\gamma}_\eps u + \Delta u\big\|_{L^p(Q\R^n_\gamma)}^p 
			\\
			&= \int_{Q \R^n_\gamma}\left| 
			\,\mathrm{P.V.} \int_{Q \R^n_\gamma} J_\eps(x-y)\big(u(x)-u(y)\big)\dy 
			+ \Delta u(x)\right|^p\dx 
			\\
			&= \int_{\R^n_\gamma}\left|
			\,\mathrm{P.V.} \int_{\R^n_\gamma}
			J_\eps\big(Q(x-y)\big)\big(w(x)-w(y)\big)\dy 
			+ \Delta w(x)\right|^p
			\dx
			\\[1ex]
			&= \big\|\widehat{\LL}^{\R^n_\gamma}_\eps w + \Delta w \big\|_{L^p(\R^n_\gamma)}^p.
		\end{split}
	\end{align}
    In particular, invoking the estimate \eqref{BND:LOM:HS} from Theorem~\ref{THM:CONV:HS} and the identity \eqref{ID:W2P}, this entails
    \begin{align}
		\label{TRAFO:Q*}
		\begin{split}
			\big\|\mathcal{L}^{Q\R^n_\gamma}_\eps u \big\|_{L^p(Q\R^n_\gamma)}
            &\le \big\|\mathcal{L}^{Q\R^n_\gamma}_\eps u 
                    + \Delta u\big\|_{L^p(Q\R^n_\gamma)}
                + \big\| \Delta u \big\|_{L^p(Q\R^n_\gamma)}
			\\
			&= \big\|\widehat{\LL}^{\R^n_\gamma}_\eps w 
                + \Delta w \big\|_{L^p(\R^n_\gamma)}
                + \big\| \Delta w \big\|_{L^p(\R^n_\gamma)}
            \\
			&\le  \big\|\widehat{\LL}^{\R^n_\gamma}_\eps w 
                \big\|_{L^p(\R^n_\gamma)}
                + 2\big\| \Delta w \big\|_{L^p(\R^n_\gamma)}
            \\
            &\le C \| w \|_{W^{2,p}(\R^n_\gamma)}
            = C \| u \|_{W^{2,p}(Q\R^n_\gamma)}.
		\end{split}
	\end{align}    
    Consequently, $\mathcal{L}^{Q\R^n_\gamma}_\eps$ can be extended to a bounded linear operator
    \begin{align*}
        \LL_\eps^{Q\R^n_\gamma}: 
        W^{2,p}_Q(Q\R^n_\gamma) \to L^p(Q\R^n_\gamma) \,.
    \end{align*}

    Now, we additionally assume that $u\in W^{3,p}_Q(Q\R^n_\gamma)$. Hence,
    $u\in W^{3,p}_I(\R^n_\gamma)$ with 
    \begin{equation}
        \label{ID:W3P}
        \|w\|_{W^{3,p}(\R^n_\gamma)} = \|u\|_{W^{3,p}(Q\R^n_\gamma)}.
    \end{equation} 
    Using the inequality \eqref{CONV:CHS} from Theorem~\ref{THM:CONV:HS} to estimate the right-hand side of \eqref{TRAFO:Q}, we obtain
    \begin{align*}
        &\big\|\mathcal{L}^{Q\R^n_\gamma}_\eps u + \Delta u\big\|_{L^p(Q\R^n_\gamma)} 
        = \big\|\widehat{\LL}^{\R^n_\gamma}_\eps w 
            + \Delta w \big\|_{L^p(\R^n_\gamma)}
        \\
        &\quad \le C_\gamma\, \sqrt[p]{\eps}\, \|w\|_{W^{3,p}(\R^n_\gamma)}
        = C_\gamma\, \sqrt[p]{\eps}\, \|u\|_{W^{3,p}(Q\R^n_\gamma)}
    \end{align*}
    as desired. Proceeding as in Step~3 of the proof of Theorem~\ref{THM:CONV:HS}, we apply the Banach--Steinhaus theorem to conclude that
    \begin{align*}
        \mathcal{L}^{Q\R^n_\gamma}_\eps u \to -\Delta u
        \quad\text{for all $u\in W^{2,p}(Q\R^n_\gamma)$.}
    \end{align*}    
    Therefore, the proof of Corollary~\ref{COR:CONV:HS} is complete.
\end{proof}


\subsection{Convergence on a domain with compact boundary} \label{SUBSEC:DOMAIN}

In this section, we finally consider the case where $\Omega$ is a domain with compact $C^3$-boundary. 

\medskip

\begin{theorem}\label{THM:CONV:DOM}
	Let $\Omega\subset\R^n$ be a domain with compact boundary of class $C^3$, $p\in [1,\infty)$, $\eps\in (0,1]$, and suppose that \ref{ASS:RHO} and \ref{ASS:J} hold. 
	Then, the operator introduced in \eqref{DEF:LOM:CC2L1LOC} $\big($restricted to $C^2_c(\overline{\Omega}) \cap W^{2,p}_M(\Omega)\big)$ can be extended to a bounded linear operator
    \begin{equation}
        \label{DEF:LOM:DOM}
        \LL_\eps^{\Omega}: W^{2,p}_M(\Omega) \to L^p(\Omega)
    \end{equation}
    with
    \begin{equation}
        \label{BND:LOM:DOM}
        \norm{\LL_\eps^{\Omega} u}_{L^p(\Omega)} 
        \leq c_\Omega \|u\|_{W^{2,p}(\Omega)}
    \end{equation}
    for all $u\in W^{2,p}_M(\Omega)$, where $c_\Omega$ is a positive constant depending only on $\Omega$, $\rho$ and $p$.
    Moreover, there exists a positive constant $C_\Omega$ depending only on $\Omega$, $\rho$ and $p$ such that for all $\eps>0$ and all $u\in W^{3,p}_M(\Omega)$, it holds
	\begin{equation}
		\label{CONV:DOM}
		\big\|\LL_\eps^{\Omega} u - \LL^{\Omega} u \big\|_{L^p(\Omega)} 
        \leq C_\Omega\, \sqrt[p]{\eps}\, \|u\|_{W^{3,p}(\Omega)}.
	\end{equation}
    Furthermore, for all $u\in W^{2,p}_M(\Omega)$, it holds
    \begin{equation}
		\label{CONV:DOM:2}
		\LL_\eps^{\Omega} u \to \LL^{\Omega} u 
        \quad\text{in $L^p(\Omega)$ as $\eps\to 0$.}
	\end{equation}
\end{theorem}

\medskip

\begin{proof}
    To provide a cleaner presentation, we will usually refrain from indicating the principal value by the symbol $\mathrm{P.V.}$ whenever the meaning is clear. 
    In the following, the letter $C$ will denote generic positive constants depending only on $\Omega$, $\rho$ and $p$. The concrete value of $C$ may vary throughout the proof.
    The proof is split into three steps.
	
	\noindent\textbf{Step~1. Reduction to a simpler setting on a reference domain.}\\
    We already know from Corollary~\ref{COR:WD:DOM} that the operator
    \begin{align}
    \label{DEF:LOM:DOM*}
        \begin{split}
            &\LL_\eps^\Omega: C_c^2(\overline\Omega) \cap W^{2,p}_M(\Omega)
            \to L^p_{\mathrm{loc}}(\Omega),
            \\
            &\LL_\eps^\Omega v(x) 
            = \mathrm{P.V.} \int_{\Omega} J_\eps(x-y) \big( v(x) - v(y) \big) \dy
        \end{split}
    \end{align}
    is well-defined and linear.
    Let now $u\in C^2_c(\overline\Omega)\cap W^{2,p}_M(\Omega)$ be arbitrary and let $A$ be the matrix introduced in \eqref{DEF:A}. This means that
    \begin{equation*}
        \det A 
        = \det\sqrt{M}
        = (\det M)^{\frac12}.
    \end{equation*} 
	Using the change of variables $z\mapsto Az$, we observe that
	\begin{align*}
		AA^T &=  M 
		= \frac{1}{2} \int_{\Omega}J(z)z\otimes z\dz 
		=  \frac{1}{2} A\Bigg(\int_{A^{-1}\Omega}J(Az)z\otimes z\dz\Bigg)A^T \det A .
	\end{align*} 
	This implies that
	\begin{align}
		\label{TRAFO:1}
		\frac{1}{2} \int_{\R^n}J(Az)z\otimes z\dz = I (\det A)^{-1},
	\end{align}
	where $I\in \R^{n\times n}$ denotes the identity matrix.
	In the following, we choose $\widehat\Omega \coloneqq A^{-1}\Omega$ as our reference domain. 
    
    As $\Omega$ is a domain of class $C^3$, so is $\widehat\Omega$. Thanks to the chain rule, the normal spaces satisfy the relation
	\begin{align*}
		\mathrm{N}_{\hat{x}}(\partial\widehat\Omega) 
		= A^T \mathrm{N}_{A\hat{x}}(\partial\Omega).
	\end{align*}
	In particular, since $A^T = A$, this implies that
	\begin{align*}
		\mathbf{n}_{\partial\widehat\Omega}(\hat{x}) = A\:\mathbf{n}_{\partial\Omega}(A\hat{x})\quad\text{ for all }x\in\partial\widehat\Omega.
	\end{align*} 	
	We now define the function
	\begin{equation*}
		w:\widehat\Omega\to\R, \quad v(\hat{x}) = u(A\hat{x}).
	\end{equation*}
	By means of the chain rule, it is straightforward to check that
    \begin{equation}
        \label{EST:DOM:WU}
        w \in C^2_c\big(\,\overline{\widehat\Omega}\,\big) 
        \cap W^{2,p}(\widehat\Omega)
        \quad\text{with}\quad
        \|w\|_{W^{2,p}(\widehat\Omega)}
        \le C \|u\|_{W^{2,p}(\Omega)} \,.
    \end{equation}
	Recalling the definition of $W^{2,p}_M(\Omega)$ (see~\eqref{DEF:WKM}) and using the above identities along with the chain rule, we observe that
	\begin{align*}
		&\nabla w(\hat{x})\cdot\mathbf{n}_{\partial\widehat\Omega}(\hat{x}) 
		= A\nabla u(A\hat{x})\cdot A\mathbf{n}_{\partial\Omega}(A\hat{x}) \\
		&\quad= A^2\nabla u(A\hat{x})\cdot \mathbf{n}_{\partial\Omega}(A\hat{x}) 
		= M\nabla u(A\hat{x})\cdot \mathbf{n}_{\partial\Omega}(A\hat{x})
		= 0
	\end{align*}
	for all $x\in\partial\widehat\Omega$. This means that $w\in W^{2,p}_I(\widehat\Omega)$.
	Moreover, by means of the chain rule, we further deduce that
	\begin{align*}
		&\Delta w(\hat{x}) = \text{div}(\nabla w(\hat{x})) 
		= \Div(A\nabla u(A\hat{x})) 
	\end{align*}
	for all $\hat{x}\in\widehat\Omega$. 
	We now define $\widehat{\rho}:\R^n \to\R$ with
	\begin{equation}
        \label{DEF:HRHO}
		\widehat{\rho}(x) = \rho(Ax) \frac{\abs{x}^2}{\abs{Ax}^2}
        \,\det A
		\quad\text{for all $x\in\R^n\setminus\{0\}$}, 
	\end{equation}
	and we write $\widehat{J}$ to denote its associated kernel. 
	Since $\rho$ and $J$ satisfy \ref{ASS:RHO} and \ref{ASS:J}, it is straightforward to check that $\widehat{\rho}$ and $\widehat{J}$ fulfill \ref{ASS:RHO} and \ref{ASS:J} with
    \allowdisplaybreaks
	\begin{align}
		\widehat M &\coloneqq \frac{1}{2} \int_{\R^n} \widehat{J}(z)\, z \otimes z \dz =  I,
		\\
		\widehat A &\coloneqq  I
	\end{align}
    \allowdisplaybreaks[0]
	in place of $M$ and $A$, respectively. 
	In particular, for all $\eps>0$, we have
	\begin{align*}
		\widehat{J}(x) = J(Ax) \,\det A
		\quad\text{and}\quad
		\widehat{J_\eps}(x) = J_\eps(Ax) \,\det A
		\quad\text{for all $x\in\R^n\setminus\{0\}$,}
	\end{align*}
	and in view of assumption \eqref{COND:J:2}, we have
	\begin{align}
		\label{Cond:kernel}
		&\int_{\R^{n-1}} \widehat{J}\left(Q
		\begin{pmatrix}
			z' \\ z_n
		\end{pmatrix}\right) z' \,\mathrm{d}z' = 0
		\quad\text{for almost all $z_n\in\R$ and all $Q\in \mathrm{O}(n)$.}
	\end{align}
	This condition will be crucial in Step~2 of this proof and it is the reason, why \eqref{COND:J:2} was demanded in Assumption~\ref{ASS:J} in the first place.

	Now, we further introduce the nonlocal operator
	\begin{align}
        \label{DEF:HATLLE}
        \begin{split}
		&\widehat{\LL}_\eps^{\widehat\Omega}: 
        C^2_c\big(\,\overline{\widehat\Omega}\,\big) 
        \cap W^{2,p}(\widehat\Omega) 
        \to L^p_\mathrm{loc}(\widehat\Omega),
		\\
		&\widehat{\LL}_\eps^{\widehat\Omega} v(x) = \mathrm{P.V.}\int_{\widehat\Omega} \widehat{J}_\eps(x-y) \big( v(x) - v(y) \big) \dy.
        \end{split}
	\end{align}
	Since $\widehat{\rho}$ and $\widehat{J}$ satisfy \ref{ASS:RHO} and \ref{ASS:J}, we already know from Corollary~\ref{COR:WD:DOM} that this operator is well-defined and linear.
    In the remainder of this proof, we need to establish the following key estimates:
    \begin{align}
        \label{KEY:1}
        \| \mathcal{L}^{\widehat\Omega}_\eps v \|_{L^p(\widehat\Omega)}
        &\le C \| v \|_{W^{2,p}(\widehat\Omega)}
        &&\quad\text{for all $v\in C^2_c\big(\,\overline{\widehat\Omega}\,\big) 
        \cap W^{2,p}_I(\widehat\Omega)$},
        \\
        \label{KEY:2}
        \|\mathcal{L}^{\widehat\Omega}_\eps v + \Delta v\|_{L^p(\widehat\Omega)}
        &\le C \sqrt[p]{\eps} \| v \|_{W^{3,p}(\widehat\Omega)}
        &&\quad\text{for all $v\in C^3_c\big(\,\overline{\widehat\Omega}\,\big) 
        \cap W^{3,p}_I(\widehat\Omega)$}.
    \end{align}
    Once this is achieved, we use the changes of variables $x\mapsto \hat{x} = A^{-1}x$ and $y\mapsto \hat{y} = A^{-1}y$ to derive the identity 
	\begin{align}
		\label{TRAFO:2}
		\begin{split}
			&\|\mathcal{L}^{\Omega}_\eps u - \LL^{\Omega} u\|_{L^p(\Omega)}^p 
			\\
			&= \int_{A\widehat\Omega}\left| 
			\,\mathrm{P.V.} \int_{A\widehat\Omega}J_\eps(x-y)\big(u(x)-u(y)\big)\dy 
			+ \Div\big(A\nabla u(x)\big) \right|^p\dx 
			\\
			&= \det A\, \int_{\widehat\Omega}\left|
			\,\mathrm{P.V.} \int_{\widehat\Omega}
			J_\eps\big(A(\hat{x}-\hat{y})\big)\,\det A \,\big(w(\hat{x})-w(\hat{y})\big)\dhy 
			+ \Delta w(\hat{x})\right|^p
			\dhx 
			\\[1ex]
			&= \det A\, \|\widehat{\mathcal{L}}^{\widehat\Omega}_\eps w + \Delta w\|_{L^p(\widehat\Omega)}^p.
		\end{split}
	\end{align}
    In particular, using \eqref{KEY:1} and \eqref{EST:DOM:WU}, we infer that
    \allowdisplaybreaks
    \begin{align*}
        \|\mathcal{L}^{\Omega}_\eps u \|_{L^p(\Omega)}
        &\le \|\mathcal{L}^{\Omega}_\eps u - \LL^{\Omega} u\|_{L^p(\Omega)}
            + \|\LL^{\Omega} u\|_{L^p(\Omega)}
        \\
        &\le \|\widehat{\mathcal{L}}^{\widehat\Omega}_\eps w 
            + \Delta w\|_{L^p(\widehat\Omega)}
            + \|\Delta w\|_{L^p(\widehat\Omega)}
        \\
        &\le \|\widehat{\mathcal{L}}^{\widehat\Omega}_\eps w 
            \|_{L^p(\widehat\Omega)}
            + 2\|\Delta w\|_{L^p(\widehat\Omega)}
        \\
        &\le C \| w \|_{W^{2,p}(\widehat\Omega)}
        \le C \| u \|_{W^{2,p}(\Omega)}.
    \end{align*}
    \allowdisplaybreaks[0]
    Consequently, $\mathcal{L}^{\Omega}_\eps$ can be extended to a bounded linear operator 
    \begin{equation*}
        \LL_\eps^\Omega:  W^{2,p}_M(\Omega)
            \to L^p(\Omega)\,,
    \end{equation*}
    which fulfills \eqref{BND:LOM:DOM}.
    Now, we take an arbitrary $u\in C^3_c(\overline\Omega)\cap W^{3,p}(\Omega)$. This means that
    \begin{equation}
        \label{EST:DOM:WU*}
        w \in C^3_c\big(\,\overline{\widehat\Omega}\,\big) 
        \cap W^{3,p}(\widehat\Omega)
        \quad\text{with}\quad
        \|w\|_{W^{3,p}(\widehat\Omega)}
        \le C \|u\|_{W^{3,p}(\Omega)} \,.
    \end{equation}
    Combining \eqref{KEY:2}, \eqref{TRAFO:2} and \eqref{EST:DOM:WU}, we further have
    \allowdisplaybreaks
    \begin{align*}
        &\|\mathcal{L}^{\Omega}_\eps u 
            - \LL^{\Omega} u\|_{L^p(\Omega)}
        = \|\widehat{\mathcal{L}}^{\widehat\Omega}_\eps w 
            + \Delta w\|_{L^p(\widehat\Omega)}^p
        \\
        &\quad\le C \sqrt[p]{\eps} \| w \|_{W^{3,p}(\widehat\Omega)}
        \le C \sqrt[p]{\eps} \| u \|_{W^{3,p}(\Omega)}.
    \end{align*}
    \allowdisplaybreaks[0]
    Because of density, it follows that \eqref{CONV:DOM} holds. 
    Proceeding as in Step~3 of the proof of Theorem~\ref{THM:CONV:HS}, we apply the Banach--Steinhaus theorem to conclude that \eqref{CONV:DOM:2} is fulfilled.
    This means that all the claims are established.
		
    Therefore, in order to complete the proof, it remains to establish the key estimates \eqref{KEY:1} and \eqref{KEY:2}.
    
	\medskip
	
	\noindent\textbf{Step~2. Localization of the reference domain.}\\
    In order to verify the key estimates \eqref{KEY:1} and \eqref{KEY:2}, we need to localize the reference domain $\widehat\Omega$.
    For this purpose, we need to construct a suitable partition of unity.

    We start by fixing an arbitrary $z \in \partial\widehat\Omega$. Then, since $\partial\widehat\Omega$ is of class $C^3$, there exists an open set $U_z \subseteq \R^n$, a function $\gamma_z\in C^3_b(\R^{n-1})$ and a matrix $Q_z\in \mathrm{SO}(n)$ such that
    \begin{align*}
        \widehat\Omega \cap U_z
        = Q_z\R^n_{\gamma_z} \cap U_z.
    \end{align*}
    Since $\gamma_z\in C^{3}_b(\R^{n-1}) $, we infer from \cite[Lemma 2.1]{Schumacher2009} that there exists a $C^{2,1}$-diffeomorphism $F_z:\R^n \to \R^n$ with $F_{\gamma_z}(\R^n_+) = \R^n_{\gamma_z}$, $F_z(x^\prime,0) = \big(x^\prime,\gamma(x^\prime) \big)$ and 
	\begin{equation}
		\label{COND:FZ}
        \big(DF_z(x',0)\big)^{-1} 
            \,\n_{\partial\R^n_{\gamma_z}}\big(F_z(x',0)\big)
        \in \mathrm{span}\{e_n\}
	\end{equation}
	for all $x'\in \R^{n-1}$. For any $\delta>0$, we now introduce the functions $\phi_z\in C^\infty_c(\R^{n-1};[0,1])$ and $\psi_z\in C^\infty_c(\R;[0,1])$ with
    \begin{alignat*}{2}
        &\phi_z = 1 \;\;\text{on $B_{\delta/2}(z')$},
        \qquad
        &&\supp \phi_z \subset B_\delta(z') \subset \R^{n-1},
        \\
        &\psi_z = 1 \;\;\text{on $\big(-\tfrac{\delta}{2},\tfrac{\delta}{2}\big)$},
        \qquad
        &&\supp \psi_z \subset (-\delta,\delta) \,.
    \end{alignat*}
    Next, we define
    \begin{align*}
        \widehat\varphi_z: \R^n \to \R,
        \quad
        \widehat\varphi_z(y) = \phi_z(y') \psi_z(y_n)
        \quad\text{and}\quad
        \tilde\varphi_z \coloneqq \widehat\varphi_z \circ \big( Q_z F_{z} \big)^{-1}\,.
    \end{align*}
    Since $(Q_zF_z)^{-1}(U) \subset \R^n$ is open, we can ensure that
    \begin{align*}
        \supp\widehat\varphi_z \subset (Q_zF_z)^{-1}(U_z)
        \quad\Leftrightarrow\quad
        \supp\tilde\varphi_z \subset U_z
    \end{align*}
    by choosing $\delta_z$ sufficiently small. Moreover, as $F_z^{-1}\in C^{2,1}(\R^n)$, we infer that
    \begin{align*}
        \tilde\varphi_z \in C^2_c(\R^n;[0,1]) \cap W^{3,\infty}(\R^n).
    \end{align*}
    Let now $y\in \partial\widehat\Omega \cap U_z = \partial\big(Q_z\R^n_{\gamma_z}\big) \cap U_z$ be arbitrary. Hence, there exists $x' \in \R^{n-1}$ such that $Q_zF_z(x',0) = y$.    
    By the construction of $F_z$, we have
    \begin{align}
        \label{ID:DQFZ}
        D\big((Q_zF_z)^{-1}\big)(y)
        = \big(D(Q_zF_z)(x',0)\big)^{-1}
        = \big(DF_z(x',0)\big)^{-1} Q_z^T
    \end{align}
    Moreover, we have
    \begin{align}
        \label{ID:QN}
        Q_z^T \, \n_{\partial(Q_z \R^n_{\gamma_z})}(y)
        = \n_{\partial(\R^n_{\gamma_z})}\big(Q_z^T y\big)
        = \n_{\partial\R^n_{\gamma_z}}\big(F_z(x',0)\big)
    \end{align}
    Now, using the chain rule as well as \eqref{ID:DQFZ} and \eqref{ID:QN}, we derive the identity
    \begin{align}
        \label{ID:DPZ}
        \begin{split}
        &\nabla \tilde\varphi_z(y) \cdot \n_{\partial\widehat\Omega}(y)
        = D\tilde\varphi_z(y) \, \n_{\partial\widehat\Omega}(y)
        = D\tilde\varphi_z(y) \, \n_{\partial(Q_z \R^n_{\gamma_z})}(y)
        \\
        &\quad = D\widehat\varphi_z\big( (Q_zF_z)^{-1}(y) \big)
            \, \big(DF_z(x',0)\big)^{-1} Q_z^T 
            \,\n_{\partial(Q_z \R^n_{\gamma_z})}(y)
        \\
        &\quad = D\widehat\varphi_z\big( (Q_zF_z)^{-1}(y) \big)
            \, \big(DF_z(x',0)\big)^{-1} 
            \,\n_{\partial\R^n_{\gamma_z}}\big(F_z(x',0)\big)
        \end{split}
    \end{align}
    By the construction of $\widehat\varphi_z$, we have
    \begin{align}
        \label{EQ:DPZ}
        D\widehat\varphi_z(y) = 
        \left(\begin{array}{ccc|c}
		& & & 0 \\
		& D\phi_z(y') & & \vdots \\
		& & & 0 \\ \hline
		0 & \ldots & 0 & \psi_z'(y_n)
		\end{array}\right)
        \quad\text{for all $y\in\R^n$}
    \end{align}
    Consequently, since $\psi_z'(0) = 0$, it holds that
    \begin{align*}
        D\widehat\varphi_z\big( (Q_zF_z)^{-1}(y) \big) e_n 
        = D\widehat\varphi_z( x',0 ) e_n
        = 0
    \end{align*}
    Thus, in view of \eqref{COND:FZ}, we conclude from \eqref{ID:DPZ} that
    \begin{align*}
        \nabla \tilde\varphi_z(y) \cdot \n_{\partial\widehat\Omega}(y)
        = 0
    \end{align*}
    for all $y\in \partial\widehat\Omega \cap U_z$.

    By the above construction, it is clear that the sets $\supp\tilde\varphi_z$, $z\in \partial\widehat\Omega$, are an open cover of $\partial\widehat\Omega$. Hence, since $\partial\widehat\Omega$ is compact, we can select $z_1,\ldots,z_N$ such that 
    \begin{align*}
        \partial\widehat\Omega 
        \subset \bigcup_{j=1}^N \{\tilde\varphi_{z_j} > 0 \}
        \subset \bigcup_{j=1}^N U_{z_j}.
    \end{align*}
    Therefore, in the following, we simply write
    \begin{align*}
        U_j \coloneqq U_{z_j}, \quad
        Q_j \coloneqq Q_{z_j}, \quad
        \gamma_j \coloneqq \gamma_{z_j}, \quad
        \tilde\varphi_j \coloneqq \tilde\varphi_{z_j}
    \end{align*}
    for all $j\in\{1,\ldots,N\}$. We further find an open set $U_0 \subset \R^n$ with $\overline{U_0} \subset \widehat\Omega$ and a function $\tilde\varphi_0\in C^\infty_c(\R^n;[0,1])$ with $\supp \tilde\varphi_0 \subset U_0$ such that
    \begin{align*}
        \overline{\widehat\Omega} 
        \subset \bigcup_{j=0}^N \{\tilde\varphi_j > 0 \}
        \subset \bigcup_{j=0}^N U_{j}.
    \end{align*}
    To simplify the notation, we further set $Q_0 \coloneqq I$ and $\R^n_{\gamma_0} = \R^n$.
    Finally, we define 
    \begin{align*}
        \varphi_j:\R^n\to\R^n,
        \quad
        \varphi_j(x) = 
        \begin{cases}
            \dfrac{\tilde\varphi_j(x)}{S(x)} &\text{if $S(x)>0$}\,,\\
            0 &\text{if $S(x)=0$}
        \end{cases}        
    \end{align*}
    for all $j\in\{0,\ldots,N\}$, where
    \begin{align*}
        S(x) = \sum_{j=0}^N \tilde\varphi_j(x)
        \quad\text{for all $x\in \R^n$.}
    \end{align*} 
    This ensures that 
    \begin{align*}
        \sum_{j=0}^N \varphi_j = 1 
        \quad\text{on $\overline{\widehat\Omega}$.}
    \end{align*}
    
    By the above construction, the family
    \begin{align*}
        \{\varphi_j\}_{j=0,...,N} \subset C^2_c(\R^n;[0,1])\cap W^{3,\infty}(\R^n)
    \end{align*}    
    is a partition of unity of $\widehat\Omega$, which satisfies
    \begin{align}
        \label{COND:POU}
        \widehat\Omega \cap U_j = Q_j\R^n_{\gamma_j} \cap U_j,
        \quad
        \supp \varphi_j \subset U_j,
        \quad
        \nabla\varphi_j \cdot \mathbf{n}_{\partial\widehat\Omega} = 0 
        \;\;\text{on $\partial\widehat\Omega$}
    \end{align}
    for all $j\in\{0,\ldots,N\}$.

    \medskip
    \noindent
    \textbf{Step~3. Proof of the key estimates \eqref{KEY:1} and \eqref{KEY:2}.}\\
    Let $v\in  C^2_c\big(\,\overline{\widehat\Omega}\,\big) 
    \cap W^{2,p}_I(\widehat\Omega)$ be arbitrary.
    For every $j\in\{0,\ldots,N\}$, we further choose a function
    \begin{equation*}
        \psi_j\in C^\infty_0(\R^n;[0,1])
        \quad\text{with}\quad
        \text{supp }\psi_j\subseteq U_j
        \quad\text{and}\quad
        \psi_j = 1 \;\;\text{on $\text{supp }\varphi_j$}.
    \end{equation*}
    Now, we set $v_j \coloneqq v\varphi_j$ for $j=0,\ldots,N$.
    Using the product rule along with \eqref{COND:POU}, we find that  
    \begin{align*}
        v_j &\in  C^2_c\big(\overline{Q_j\R^n_{\gamma_j}}\big) 
        \cap W^{2,p}_I(Q_j\R^n_{\gamma_j})
        \quad\text{for all $j\in\{0,\ldots,N\}$.}
    \end{align*}
    Recalling that $\widehat{\LL}_\eps^{\widehat\Omega}$ is linear, we
    use the decomposition $v = \sum_{j=0}^N v_j$, to obtain
    \begin{align}
        \label{DEC:LLE:1}
        \widehat{\LL}_\eps^{\widehat\Omega} v
        &= \sum_{j=0}^N \widehat{\LL}_\eps^{\widehat\Omega}v_j
        = \sum_{j=0}^N \Big( J^\eps_{j,1}
            + J^\eps_{j,2} \Big).
    \end{align}
    with
    \begin{align}
        \label{DEF:JE:12}
        J^\eps_{j,1} 
        \coloneqq \psi_j \,\widehat{\LL}_\eps^{\widehat\Omega}v_j,
        \qquad
        J^\eps_{j,2} 
        \coloneqq (1-\psi_j)\, \widehat{\LL}_\eps^{\widehat\Omega}v_j
    \end{align}
    for $j=0,\ldots,N$. We will now estimate these summands separately.

    \noindent\textit{Ad $J^\eps_{j,1}$:}
    Let $j\in\{0,\ldots,N\}$ be arbitrary.
    We obtain
    \begin{align*}
    J^\eps_{j,1}(x) 
    &= \psi_j(x)\int_{{\widehat\Omega}}\widehat{J}_\varepsilon(x-y)\big(v_j(x)- v_j(y)\big)\dy
    \\
    &= J^\eps_{j,1,1} + J^\eps_{j,1,2} - J^\eps_{j,1,3}
    \end{align*}
    with 
    \allowdisplaybreaks
    \begin{align*}
        J^\eps_{j,1,1} (x) 
        &\coloneqq 
        \psi_j(x)\int_{Q_j\R^n_{\gamma_j}}
        \widehat{J}_\varepsilon(x-y)\,\big(v_j(x)- v_j(y)\big)\dy,
        \\
        J^\eps_{j,1,2} (x)
        &\coloneqq
        \psi_j(x)\int_{{\widehat\Omega}\setminus U_j}\widehat{J}_\varepsilon(x-y)\,v_j(x)\dy,
        \\
        J^\eps_{j,1,3} (x)
        &\coloneqq 
        \psi_j(x)\int_{Q_j\R^n_{\gamma_j}\setminus U_j}\widehat{J}_\varepsilon(x-y)\,v_j(x)\dy
    \end{align*}
    \allowdisplaybreaks[0]
    for all $x\in\widehat\Omega$. 
    Since $\widehat{\rho}$ and $\widehat{J}$ satisfy \ref{ASS:RHO} and \ref{ASS:J}, 
    Theorem~\ref{THM:CONV:R^N} (for $j=0$) and
    Corollary~\ref{COR:CONV:HS} (for $j>0$) imply that
    \begin{equation}
        \label{EST:JJ11:W2}
        \| J^\eps_{j,1,1} \|_{L^p(\widehat\Omega)}
        \le C \| v_j \|_{W^{2,p}(\widehat\Omega)}
        \le C \| v \|_{W^{2,p}(\widehat\Omega)}.
    \end{equation}
    We further recall that $\supp v_j \subseteq \supp\varphi_j \subset U_j$. Thus, for $x\in \supp\varphi_j$ and $y\in \R^n\setminus U_j$, it holds that
    \begin{equation*}
        |x-y| \ge \delta_j
        \quad\text{with}\quad
        \delta_j \coloneqq \dist(\supp \varphi_j,\R^n\setminus U_j).
    \end{equation*}
    Moreover, invoking the definition of $\widehat\rho$ (see~\eqref{DEF:HRHO}) and Lemma~\ref{LEM:DIRAC}\ref{LEM:DIRAC:B}, a straightforward computation yields
    \begin{align*}
        | J^\eps_{j,1,2}(x) |
        &\le \delta_j^{-3}\, |v_j(x)|\,  \int_{\R^n} \widehat\rho_\eps(x-y) \, |x-y| \dy
        \\
        &= \eps\,\delta_j^{-3}\, |v_j(x)|\,  \int_{\R^n} \widehat\rho(x-y) \, |x-y| \dy
        \\
        &\le C\eps\,\delta_j^{-3}\, |v_j(x)|\,. 
    \end{align*}
    Consequently, we obtain
    \begin{align}
        \label{EST:JJ12:W2}
        \| J^\eps_{j,1,2} \|_{L^p(\widehat\Omega)}
        \le C\eps \| v_j \|_{L^{p}(\widehat\Omega)}
        \le C\eps \| v \|_{L^{p}(\widehat\Omega)}.
    \end{align}
    The term $J^\eps_{j,1,3}$ can be estimated completely analogously and we end up with
    \begin{align}
        \label{EST:JJ13:W2}
        \| J^\eps_{j,1,3} \|_{L^p(\widehat\Omega)}
        \le C\eps \| v \|_{L^{p}(\widehat\Omega)}.
    \end{align}
    Combining \eqref{EST:JJ1:W2} and \eqref{EST:JJ2:W2} and recalling that $\eps\in(0,1]$, we conclude that
    \begin{align}
        \label{EST:JJ1:W2}
        \| J^\eps_{j,1} \|_{L^p(\widehat\Omega)}
        \le C \| v \|_{W^{2,p}(\widehat\Omega)}.
    \end{align}
    for all $j\in\{0,\ldots,N\}$.

    \medskip\pagebreak[4]

    \noindent\textit{Ad $J^\eps_{j,2}$:}
    Let $j\in\{0,\ldots,N\}$ be arbitrary.
    We first notice that $1-\psi_j(x) = 0$ if $x\in \supp \varphi_j$. If $x\in \widehat\Omega\setminus \supp \psi_j$ and $y\in \supp \varphi_j$, we have
    \begin{equation*}
        |x-y| \ge \tilde\delta_j
        \quad\text{with}\quad
        \tilde\delta_j \coloneqq \dist(\R^n\setminus \supp\psi_j,\supp \varphi_j).
    \end{equation*}
    Hence, for all $x\in\widehat\Omega$, it holds that
    
    \begin{align*}
        |J^\eps_{j,2}(x)|
        \le \big(1 - \psi_j(x) \big) 
        \tilde\delta_j^{-3}\, \,  \int_{\widehat\Omega} \widehat\rho_\eps(x-y) \, |x-y| |v_j(y)|\dy\,.
    \end{align*}
    This directly implies 
    \begin{align}
        \label{EST:JJ2:W2}
        \| J^\eps_{j,2} \|_{L^p(\widehat\Omega)}^p 
        &\le C\tilde\delta_j^{-3p}\int_{\R^n}\Bigg(\int_{\widehat\Omega}
        \widehat\rho_\eps(x-y) \, |x-y| \dy\Bigg)^{p-1}
        \nonumber\\
        &\qquad\qquad\qquad
        \cdot
        \Bigg(\int_{\widehat\Omega}\widehat\rho_\eps(x-y) \, |x-y| |v_j(y)|^p\dy\Bigg)\dx 
        \nonumber\\
        &\quad \le C\tilde\delta_j^{-3p}\eps^{p-1}\int_{\widehat\Omega}\left(\int_{\R^n}\widehat\rho_\eps(x-y) \, |x-y| \dx\right)|v_j(y)|^p\dy 
        \nonumber\\
        &\quad \le C\eps^p \| v_j \|_{L^{p}(\widehat\Omega)}^p
        \le C\eps^p \| v \|_{L^{p}(\widehat\Omega)}^p.
    \end{align}
    \medskip

    In summary, combining \eqref{EST:JJ1:W2} and \eqref{EST:JJ2:W2} to bound the right-hand side of \eqref{DEC:LLE:1}, we infer the estimate
    \begin{align*}
        \| \widehat{\LL}_\eps^{\widehat\Omega}v_j \|_{L^p(\widehat\Omega)}
        &\le C\| v \|_{W^{2,p}(\widehat\Omega)}
    \end{align*}
    for all $v\in  C^2_c\big(\,\overline{\widehat\Omega}\,\big) 
    \cap W^{2,p}_I(\widehat\Omega)$. This means that key estimate \eqref{KEY:1} is established.

    In order to verify key estimate \eqref{KEY:2}, let now $v\in  C^3_c\big(\,\overline{\widehat\Omega}\,\big) 
    \cap W^{3,p}_I(\widehat\Omega)$ be arbitrary.
    Using the decomposition \eqref{DEC:LLE:1} along with the linearity of the Laplacian, we derive the identity
    \begin{align}
        \label{DEC:LLE:2}
        \widehat{\LL}_\eps^{\widehat\Omega} v + \Delta v
        &= \sum_{j=0}^N \Big ( \big(J^\eps_{j,1,1}  + \Delta v_j\big)        
            + J^\eps_{j,1,2}
            + J^\eps_{j,1,3} 
            + J^\eps_{j,2} \Big)\,,
    \end{align}
    where $J^\eps_{j,1,1}$, $J^\eps_{j,1,2}$, $J^\eps_{j,1,3}$ and $J^\eps_{j,2}$, $j=0,\ldots,N$, are defined as above.
    Since $\psi_j = 1$ on $\supp v_j$, we have
    \begin{align*}
        J^\eps_{j,1,1}  + \Delta v_j
        = \psi_j(x) \left( \int_{Q_j\R^n_{\gamma_j}}
        \widehat{J}_\varepsilon(x-y)\,\big(v_j(x)- v_j(y)\big)\dy
        + \Delta v_j \right).
    \end{align*}
    Hence, applying Theorem~\ref{THM:CONV:R^N} (for $j=0$) and
    Corollary~\ref{COR:CONV:HS} (for $j>0$), we deduce
    \begin{align}
        \label{EST:JJ11:W3}
        \| J^\eps_{j,1,1}  + \Delta v_j \|_{L^p(\widehat\Omega)}
        \le C\sqrt[p]{\eps}\, \| v_j \|_{W^{3,p}(\widehat\Omega)}
        \le C\sqrt[p]{\eps}\, \| v \|_{W^{3,p}(\widehat\Omega)}.
    \end{align}
    In combination with \eqref{EST:JJ12:W2}, \eqref{EST:JJ12:W2} and \eqref{EST:JJ2:W2}, we conclude from \eqref{DEC:LLE:2} that
    \begin{align*}
        \| \widehat{\LL}_\eps^{\widehat\Omega} v + \Delta v \|_{L^p(\widehat\Omega)}
        \le C\sqrt[p]{\eps}\, \| v \|_{W^{3,p}(\widehat\Omega)}.
    \end{align*}
    This means that key estimate \eqref{KEY:2} is established and thus, the proof of Theorem~\ref{THM:CONV:DOM} is complete.
\end{proof}


\footnotesize

\bibliographystyle{abbrv}
\bibliography{bibliography}

\end{document}